\documentclass[12pt,reqno]{amsart}
\usepackage{amsmath,amsxtra,amssymb,amsthm,amsfonts,eufrak}

\usepackage{tikz-cd}
\makeatletter
\newcommand*{\rom}[1]{\expandafter\@slowromancap\romannumeral #1@}
\makeatother
\usetikzlibrary{matrix,arrows,decorations.pathmorphing}

\newcommand{\R}{{\mathcal R}}

\newcommand{\C}{{\mathbb C}}
\newcommand{\ten}{\otimes}

\newcommand{\pl}{\hspace{.1cm}}

\newcommand{\ran}{\rangle}
\newcommand{\lan}{\langle}
\newcommand{\al}{\alpha}
\newcommand{\si}{\sigma}

\newcommand{\la}{\lambda}

\newcommand{\F}{{\mathcal F}}
\newcommand{\E}{{\mathcal E}}

\newcommand{\M}{{\mathcal M}}

\renewcommand{\S}{{\mathcal S}}

\renewcommand{\L}{\mathcal{L}}

\newcommand{\N}{{\mathcal N}}

\newcommand{\norm}[2]{\parallel \! #1 \! \parallel_{#2}}

\newtheorem{lemma}{Lemma}[section]
\newtheorem{prop}[lemma]{Proposition}
\newtheorem{theorem}[lemma]{Theorem}
\newtheorem{cor}[lemma]{Corollary}

\newtheorem{rem}[lemma]{Remark}

\newcommand{\re}{\begin{rem}\rm}
\newcommand{\mar}{\end{rem}}
\newtheorem{exam}[lemma]{Example}
\newcommand{\bra}[1]{\langle{#1}|}
\newcommand{\ket}[1]{|{#1}\rangle}
\newcommand{\ketbra}[1]{|{#1}\rangle\langle{#1}|}

\newcommand{\qd}{\end{proof}\vspace{0.5ex}}
\newcommand{\prf}{\begin{proof}[\bf Proof:]}

\newcommand{\xspace}{\hbox{\kern-2.5pt}}

\newtheorem{thmx}{Theorem}

\allowdisplaybreaks

\begin{document}
\title{Relative entropy for von Neumann subalgebras}
\author{Li Gao}
\address{Department of Mathematics\\
Texas A\&M University, College Station, TX 77840, USA} \email[Li Gao]{ligao@tamu.math.edu}
\author{Marius Junge$^*$}\thanks{$^*$ Partially supported by NSF grant DMS-1839177  and  DMS-1800872}
\address{Department of Mathematics\\
University of Illinois, Urbana, IL 61801, USA} \email[Marius Junge]{mjunge@illinois.edu}
\author{Nicholas LaRacuente}
\address{Department of Physics, University of Illinois, Urbana, IL 61801, USA} \email[Nicholas LaRacuente]{laracue2@illinois.edu}
\maketitle
\begin{abstract}We revisit the connection between index and relative entropy for an inclusion of finite von Neumann algebras. We observe that the Pimsner-Popa index connects to sandwiched $p$-R\'enyi relative entropy for all $1/2\le p\le \infty$, including Umegaki's relative entropy at $p=1$. Based on that, we introduce a new notation of relative entropy to a subalgebra which generalizes subfactors index. This relative entropy has application in estimating decoherence time of quantum Markov semigroup.
\end{abstract}

\section{Introduction}
\footnotetext{2010 \emph{Mathematics Subject Classification}: Primary: 46L53. Secondary: 46L60, 46L37, 46L51}
The index $[\M:\N]$ for a II$_1$ subfactor $\N\subset\M$ was first constructed by Jones \cite{Jones} as the coupling constant of the representation of $\N$ on $L_2(\M)$. On the other hand, motivated from classical egordic theory, Connes and St\"ormer \cite{CS75} introduced the relative entropy $H(\M|\N)$ for an inclusion of finite (dimensional) $\N\subset \M$. The connection between these two quantities was first studied by Pimsner and Popa \cite{pipo} and they proved the general relation
\begin{align}
\log [\M:\N]\ge H(\M|\N)\pl.\label{relation1}
\end{align}
A key concept in their discussion is the following index for an inclusion $\N\subset \M$ of finite von Neumann algebras (which we call Pimsner-Popa index),
\begin{align}
\la(\M:\N)=\max \{\la \pl | \pl \la \rho \le  E(\rho)\pl, \pl \pl \forall\pl  \rho\in \M_+\pl.\}
\end{align}
Here $E:\M\to \N$ is the trace preserving conditional expectation onto $\N$ and $\M_+$ is the positive cone. It was proved in \cite{pipo} that $[\M:\N]=\la(\M:\N)^{-1}$ for II$_1$ subfactors, and $\log\la(\M:\N)^{-1}\ge H(\M|\N)$ for finite von Neumann algerbas. Form this the inequality \eqref{relation1} follows. %Moreover, $\la(\M:\N)$ was used to approximated the index for hyperfinite subfactors via finite dimensional $\N\subset\M$, for which an explicit formula for $\la(\M:\N)$ was obtained.

In this paper, we revisit these concepts and connect them to sandwiched R\'enyi relative entropies $D_p$ recently introduced in quantum information theory. Let $p\in[1/2,1)\cup (1,\infty]$ and $1/p+1/p'=1$. For two densities $\rho,\si \in \M$ (positive and trace $1$) and $\si$ invertible,
\[D_p(\rho||\si):=\frac{1}{p-1}\log tr(|\si^{-\frac{1}{2p'}}\rho \si^{-\frac{1}{2p'}}|^p)\pl.\]
where $tr$ is the trace on $\M$. $D_p$ are R\'enyi type generalization of Umegaki's relative entropy $D(\rho||\si)=tr(\rho\log \rho-\rho\log \si)$. While the $D$ commonly has operational meaning in the asymptotic i.i.d setting (e.g. \cite{hypo,WildeWang}), the sandwiched R\'enyi relative entropy $D_p$ has been found useful in proving strong converse theorems (e.g. \cite{wilde14,wilde15,wilde16}). Our starting point is the observation that the quantity $\la(\M:\N)$ is closely related to the sandwiched R\'enyi relative entropy $D_p$ at $p=\infty$. Based on that, we obtain the following connection between Popa-Pimsner index and $p$-R\'enyi relative entropy for all $1/2\le p\le \infty$, including Umegaki's relative entropy at $p=1$.
\begin{thmx}[c.f. Theorem \ref{index}]\label{A}Let $\N\subset \M$ be an inclusion of II$_1$ factors or finite dimenisional von Neumann algebras. For $1/2\le p\le \infty$,
\begin{align}-\log\la(\M:\N)=\sup_{\rho\in \M}D_{p}(\rho||E(\rho))=\sup_{\rho\in \M}\inf_{\si\in \N} D_{p}(\rho||\si)\pl,
\end{align}
where the supremum takes all density operators $\rho$ in $\M$ and the infimum takes all density operators $\si$ in $\N$.
\end{thmx}
One notation appears in above theorem is the relative entropy to the subalgebra,
\[D_p(\rho||\N)=\inf_\si D_p(\rho||\si)\] where the infimum takes all densities $\si\in \N$. As $D_p(\rho||\si)$ being an information metric between $\rho$ and $\si$, $D_{p}(\rho||\N)$ measures the distance of the state $\rho$ is from the subalgebra $\N$. $D_{p}(\rho||\N)$ unifies several information measures in quantum information including (R\'enyi) conditional entropy \cite{muller13}, relative entropy of coherence \cite{winter16} and of asymmetry \cite{marvian14}, which are important quantifier of operational resources in quantum information tasks. From noncommutative $L_p$-spaces perspective, $D_p(\rho||\N)$ correpsonds to the amalgamated $L_p$-spaces studied in \cite{JPmemo}. Theorem \ref{A} states that the Pimsner-Popa index can be viewed as the maximal relative entropy from $\M$ to the subalgebra. Motivated from that, we introduce new notations of relative entropy for an inclusion $\N\subset\M$ of finite von Neumann algebras,
 \[D_p(\M||\N):=\sup_{\rho\in \M} D_p(\rho||\N)\pl, D_{p,cb}(\M||\N):=\sup_{n} D_{p}(M_n(\M)||M_n(\N))\pl,\]
 where $M_n$ denote the $n$-dimensional matrix. These relative entropies differ with Connes-St\"ormer entropy $H(\M|\N)$ and the relative entropy discussed in \cite{seo}. In particular, $D_{p,cb}(\M||\N)=\log [\M:\N]$ for finite subfactors, and for $p=1$ and $\infty$, they satisfy the multiplicity (c.f. Theorem \ref{add})
\[D_{cb}(\M_1\overline{\ten}\M_2||\N_1\overline{\ten }\N_2)=D_{cb}(\M_1||\N_1)+D_{cb}(\M_2||\N_2)\pl.\]

One application of $D_{p,cb}$ is to estimate the decoherence time of quantum Markov semigroups. A quantum Markov semigroup $(T_t)_{t\ge 0}:\M\to \M$ is an ultra-weak continuous family of normal unital completely positive maps. When $\M=B(H)$ is the bounded operators on a Hilbert space $H$, quantum Markov semigroups are also called GLKS equations in physics literature (see \cite{GLKS}). It models the evolution of open quantum system that potentially interacts with environment. Let $\N$ be the subalgebra of the common multiplicative domain of $T_t$ for all $t\ge 0$. A semigroup $T_t$ is {\bf self-adjoint} if $T_t$ is self-adjoint map with respect to the trace; is {\bf primitive} if $\N=\mathbb{C}1$ is trivial. The non-primitive semigroup was studied in \cite{bardet} as a general model of decoherence process, which means a state $\rho$ loses its quantum coherence and becomes like a classic state $E(\rho)$. We obtain the following convergence property of self-adjoint semigroups.
\begin{thmx}[c.f. Theorem \ref{d3}]\label{B}
Let $T_t=e^{-At}:\M\to \M$ be a self-adjoint quantum Markov semigroup with generator $A$ and let $\N$ be the common multiplicative domain of $T_t$. Suppose $D_{2,cb}(\M||\N)<\infty$ and $T_t$ has $\la$-spectral gap that $\la\norm{x-E(x)}{2}^2\le tr(x^*Ax)$. Then for any density $\rho\in M_n(\M)$,
\begin{align}D(id_{M_n}\ten T_t(\rho)||M_n(\N))\le 2\exp{\Big(-\la t +\frac{1}{2}D_{2}(\rho||M_n(\N))\Big)}\pl.\label{decay1}\end{align}
For $\epsilon>0$, we have ${\norm{id\ten T_t(\rho)-id \ten E(\rho)}{1}\le \epsilon}$ if
\begin{align}t\ge \frac{1}{\la} \big(2\log\frac{2}{\epsilon}+D_{2,cb}(\M||\N)/2\big)
\label{deco1}\end{align}
\end{thmx}The decoherence time is the smallest time $t$ such that the trace distance ${\norm{T_t(\rho)-E(\rho)}{1}}$ is small than $\epsilon$. It is the analog of mixing time of classical Markov process when $\M$ is commutative. One standard approach for the decay property of relative entropy is the modified logarithmic Sobolev inequality (MLSI), which in our setting corresponds to
\begin{align} D(T_t(\rho)||\N)\le e^{-\la t}D(\rho||\N)\pl.\label{MLSI}\end{align}
The MLSI has been intensively investigated in classical case (see e.g. \cite[Chapter 5]{BGL} and the references therein) and recently has been studied for quantum Markov semigroups (e.g. \cite{CLSI,carlen1,carlen2,cao2019,datta2,bardet,cambyse, muller18}). In the classical case, an important property of MLSI is the tensorization, i.e. the decay exponent $\la$ for a tensor product semigroup $S_t\ten T_t$ is bounded by the exponents of $S_t$ and $T_t$. This property allows us to derive MLSI for composite system form components. Nevertheless, while such property is also desired for tensor product quantum system, it is not known for MLSI of quantum Markov semigroups. The tensorization requires that not only $T_t$ but also $id_{M_n}\ten T_t$ satisfies MLSI \eqref{MLSI} for a uniform constant $\la$ independent of $n$. This strong property was studied in \cite{CLSI} under the name ``complete logarithmic Sobolev inequality'' (CLSI).

Theorem \ref{B} gives a decay estimate of $D(\rho||\N)$ under the assumption of spectral gap and $D_{2,cb}(\M||\N)<\infty$. Despite the factor $D(\rho||\N)$ is generically smaller $e^{D_{2}(\rho||\N)}$, the exponent given by the spetral gap $\la$ is not less than the MLSI exponent. Theorem \ref{B} also gives an uniform decoherence time \eqref{deco1} for $id\ten T_t$ independent of the auxiliary system $M_n$. Moreover, it naturally extends to tensor product semigroup because the spectral gap has tensorization property even in the noncommutative case and $D_{cb}(\M||\N)$ is additive. One immediate consequence is that Theorem \ref{B} also estimates the loss of entanglement. A density $\rho\in M_n(\M)\cong M_n\ten \M$ is entangled if $\rho$ cannot be written as a convex combination of product densities $\rho=\sum_{j}\mu_j\omega_j\ten\rho_j$. Entanglement is an essential quantum phenomenon as well as fundamental resource in quantum information science. When $\N$ is a commutative algebra, the semigroup $id\ten T_t(\rho)$ converges to $id\ten E(\rho)$, which is always a non-entangled state. In other words, when a quantum system $\M$ decoherence to a classical system $\N$, it simultaneously lose its entanglement to auxiliary system or environment. Theorem $B$ gives a quantitive description of this phenomenon as it estimates the relative entropy and trace distance from $id_{M_n}\ten T_t(\rho)$ to the non-entangled state $id\ten E(\rho)$.

The rest of paper is organized as follows. In Section 2, we review the definition of sandwiched R\'enyi $p$-entropy and some basic properties. We also introduce $D_p(\rho||\N)$ and discuss the connection to the amalgamated $L_p$-spaces. Section 3 proves Theorem \ref{A} and some basic properties about $D_{p}(\M||\N)$ and $D_{p,cb}(\M||\N)$. Section 4 is devoted to application of $D_{p,cb}(\M||\N)$ in the decoherence time and proves Theorem \ref{B}. The maintext of this paper should be accessible to quantum information audience. We provide an appendix on amalgamated $L_p$-spaces and put some technical lemmas there.

{\bf Acknowledgement}---The authors are grateful to Gilles Pisier for helpful discussion related to Proposition \ref{unique}.

\section{Relative entropy}
\subsection{Sandwiched R\'enyi relative entropy}We denote by $\mathbb{C}$ the complex numbers and $M_n$ the $n\times n$ complex matrices. Throughout the paper, we consider $\M$ is a finite von Neumann algebra equipped with normal faithful finite  trace $tr$. For $1\le p<\infty$, the space $L_p(\M)$ is defined as the norm completion of $\M$ with respect to the $L_p$-norm
$\norm{x}{p}=tr(|x|^p)^{\frac{1}{p}}$. We identify $L_\infty(\M):=\M$ and the predual space $\M_*\cong L_1(\M)$ via the duality
\[a\in L_1(\M)\longleftrightarrow \phi_a\in \M_*,\pl  \phi_a(x)=tr(ax)\pl.\]
We say $\rho \in L_1(\M)$ is a density if $\rho\ge 0$ and $tr(\rho)=1$. The set of all densities correspond to the normal states of $\M$, which we denote by $S(\M)$. Let $p\in [\frac{1}{2},1)\cup (1,\infty]$ and ${1}/{p}+{1}/{p'}=1$. For two densities $\rho$ and $\si$, the sandwiched R\'enyi relative entropy is defined as %\footnote{Throughout the paper, we write $\log$ for logarithm function of basis $e$.}
\[D_p(\rho||\si)=\begin{cases}
                   p'\log \norm{\si^{-\frac{1}{2p'}}\rho \si^{-\frac{1}{2p'}}}{p}, & \mbox{if } \rho<<\si \\
                   +\infty, & \mbox{otherwise}.
                 \end{cases}\]
where $\rho<<\si$ means that the support projections satisfies $\text{supp}(\rho)\le \text{supp}(\si)$. The negative power $\si^{-\frac{1}{2p'}}$ is interpreted as generalized inverse on the support and in most discussion we can assume that $\si$ is faithful. This definition was introduced in \cite{wilde14, muller13} for matrices and recently generalized to general von Neumann algebras via different methods \cite{berta18,Jencova18,Jencova2,gu19}. When $p\to 1$, $D_p$ recovers the relative entropy
\begin{align}D(\rho||\si)=tr(\rho\log \rho-\rho\log \si)\label{relativeentropy}\end{align}
which was first introduced by Umegaki \cite{Umegaki62} and later extended to von Neumann algebras by Araki \cite{Araki76}. Umegaki's definition is the noncommutative generalization of Kullback-Leibler divergence form probablity theory. It is an fundamental quantity that have been intensive studied and widely used in quantum information theory (see \cite{Vedral02} for a survey). More recently, the sandwiched R\'enyi relative entropy $D_p$ has been found useful in proving strong converse theorems of communication tasks (e.g. \cite{wilde14,wilde15,wilde16}). For all $\frac{1}{2}\le p\le \infty$, $D_p(\rho||\si)$ is a measure of difference between $\rho$ and $\si$. In particular for $p=\infty$, \[D_\infty(\rho||\si)=\log \norm{\si^{-\frac{1}{2}}\rho \si^{-\frac{1}{2}}}{\infty}=\log \inf\{\la | \rho \le \la \si\}\] and $D_{\frac{1}{2}}$ is essentially the fidelity. We say $\Phi:L_1(\M)\to L_1(\M)$ is a completely positive trace preserving (CPTP) map if its adjoint $\Phi^\dag:\M\to \M$ is normal, unital and completely positive.  The CPTP maps are also called quantum channels. We summarise here some basic properties of $D_p$. For any two densities $\rho$ and $\si$,
\begin{enumerate}
\item[i)] $D_p(\rho||\si)\ge 0$. Moreover, $D_p(\rho||\si)=0 $ if and only if $\rho=\si$
\item[ii)]$D_p(\rho||\si)$ is non-decreasing over $p\in [1/2,\infty]$ and $\lim_{p\to 1}D_p(\rho||\si)=D(\rho||\si)$.
\item[iii)]For a CPTP map $\Phi:L_1(\M)\to L_1(\M)$, $D_p(\rho||\si)\ge D_p(\Phi(\rho)||\Phi(\si))$. In particular, $D_p(\rho||\si)$ is joint convex for $\rho$ and $\si$.
\end{enumerate}
i), ii) and iii) was proved in \cite{muller13,wilde14} for matrix algebras. The discussion for the case of general von Nuemann algebras can be found in \cite{berta18, Jencova18,Jencova2,gu19}.

\subsection{Relative entropy to a subalgebra}Motivated from the asymmetry measure of group in \cite{marvian14}, we introduced in \cite{gao_unifying_2017} the relative entropy to subalgebra.
Given a subalgebra $\N\subset\M$, we define for a density $\rho$,
\[D_p(\rho||\N):=\inf_{\si\in S(\N)} D_p(\rho||\si)\pl.\]
where the infimum takes over all densities $\si\in S(\N)$. This definition connects several concepts in quantum information literature:
\begin{enumerate}
\item[a)] Let $\al:G\to Aut(\M)$ be an action of a group $G$ as trace preserving $*$-automorphism of $\M$. Let $\N=\M^G:=\{ x\in \M| \al_g(x)=x \pl \forall\pl g\in G\}$ be the invariant subalgebra. Then $D_p(\rho||\M^G)$ is a $G$-asymmetry measure introduced in \cite{marvian14} and is related to the relative entropy of frameness as introduced in \cite{vaccaro08, gour09}.
\item[b)]Let $H_A,H_B$ be two finite dimensional Hilbert space.For $\M=B(H_A\ten H_B)$ and $\N=\mathbb{C}1\ten B(H_B)\subset B(H_A\ten H_B)$, $D_p(\rho||\N)$ gives the
sandwiched R\'enyi relative entropy $H_p(A|B)$ in \cite{muller13,wilde14} up to a dimension constant $D_p(\rho||\N)=H_p(A|B)_\rho+\log (\dim H_A)$. The constant comes from the fact that the induced trace on $\mathbb{C}1\ten B(H_B)\subset B(H_A\ten H_B)$ differs with the matrix $B(H_B)$ by a factor of $\dim H_A$.
\item[c)] Let $\N=\mathbb{C}^n\subset M_n=\M$ be the diagonal matrices inside the matrix algebra $M_n$. $D_p(\rho||\N)$ is the sandwiched R\'enyi relative entropy of coherence as in \cite{baumgratz14, streltsov_colloquium:_2017}.
\end{enumerate}
We have the basic properties of $D_p(\rho||\N)$ parallel to $D(\rho||\si)$.
\begin{prop}\label{basic}For $1/2\le p\le  \infty$ and a density $\rho\in S(\M)$,
\begin{enumerate}
\item[i)] $D_p(\rho||\N)\ge 0$. Moreover $D_p(\rho||\N)=0 $ if and only if $\rho\in S(\N)$
\item[ii)] $D_p(\rho||\N)$ is non-decreasing over $p\in [\frac{1}{2},\infty]$ and $\lim_{p\to 1}D_p(\rho||\N)=D(\rho||\N)$.
\item[iii)]Let $\Phi:L_1(\M)\to L_1(\M)$ be a CPTP such that $\Phi(L_1(\N))\subset L_1(\N)$. Then $D_p(\rho||\N)\ge D_p(\Phi(\rho)||\N)$. In particular, $D_p(\rho||\N)$ is convex for $\rho$.
\item[iv)] For $p=1$,
    \[ D(\rho||\N)=D(\rho||E(\rho))=H(E(\rho))-H(\rho)\]
    where $H(\rho)=-tr(\rho\log \rho)$ is the von Neumann entropy.
\end{enumerate}
\end{prop}
 \noindent i)-iii) follows from the corresponding properties of $D_p(\rho||\si)$ by taking the infimum. When $p=1$, for any density $\si\in S(\N)$,
\begin{align}\label{p} D(\rho||\si)&=tr(\rho\log \rho -\rho\log \si)=tr(\rho\log \rho)-\tau(E(\rho)\log \si)\nonumber\\&=tr(\rho\log \rho-E(\rho)\log E(\rho))-tr(E(\rho)\log \si-E(\rho)\log \nonumber E(\rho))\\&=D(\rho||E(\rho))+D(\si||E(\rho))\pl.
\end{align}
Because $D(\si||E(\rho))\ge 0$ and $D(\si||E(\rho))=0$ implies $\si=E(\rho)$, the infimum attains uniquely at $E(\rho)$. Moreover, by the condition expectation property,
\begin{align*}D(\rho \| E(\rho)) = tr(\rho \log \rho - \rho \log E(\rho)) =tr(\rho \log \rho - E(\rho) \log E(\rho))=H(E(\rho))-H(\rho)\pl.  \end{align*}
which verifies iv).

Form above properties, we see that $D_p(\rho||\N)$ are natural measures of the difference $\rho$ is from a density of $\N$. Viewing $E(\rho)$ as the projection of $\rho$, $D_p(\rho||E(\rho))$ is also a measure with respect to $\N$ and coincides with $D_p(\rho||\N)$ at $p=1$.
 Nevertheless, we note that for general $p$, $D_p(\rho||\N)\neq D_p(\rho||E(\rho))$.
\begin{exam}{\rm Let $\N\cong \mathbb{C}^2$ be the diagonal matrix in $\M=M_2$. For $0\le a\le 1$, consider the pure state  $\rho=\left[\begin{array}{cc}a& \sqrt{a(1-a)}\\ \sqrt{a(1-a)}& 1-a\end{array}\right]$. One can calculate that for $1< p\le \infty$ and $q=\frac{p}{2p-1}$,
\begin{align*} &D_p(\rho||\N)=D_p(\rho||\si_p)= p'\log (1+a^{q}(1-a)^{1-q}+(1-a)^{q}a^{1-q})\pl , \\ & \si_p=\left[\begin{array}{cc}\frac{a^{q}}{a^{q}+(1-a)^{q}}& 0\\ 0& \frac{(1-a)^{q}}{a^{q}+(1-a)^{q}}\end{array}\right]
\\
& D_p(\rho||E(\rho))=p'\log (a^{\frac{1}{p}}+(1-a)^{\frac{1}{p}}) \pl, \pl E(\rho)=\left[\begin{array}{cc}a& 0\\ 0& 1-a\end{array}\right]
\end{align*}}
So for all $1<p\le \infty$, $\si_p\neq E(\rho)$ are not the same.
\end{exam}

\subsection{Connection to amalgamated $L_p$-spaces}
The R\'enyi relative entropy $D_p(\rho||\N)$ are closely related to the amalgamated $L_p$-spaces and conditional $L_p$-spaces introduced in \cite{JPmemo}. Here we briefly recall the definitions and refer to the appendix and \cite{JPmemo} for more information. Let $1\le p\le \infty$ and $\frac{1}{p}+\frac{1}{p'}=\frac{1}{1}$. The amalgamated $L_p$-space $L_1^p(\N\subset\M)$ is the set of elements $x\in L_1(\M)$ such that $x$ admits a factorization $x=ayb$ with $a,b\in L_{2p'}(\N)$ and $y\in L_{p}(\M)$ equipped the norm
\begin{align}\norm{x}{L_1^p(\N\subset\M)} =\inf_{x=ayb}\norm{a}{L_{2p'}(\N)} \norm{y}{L_p(\M)}\norm{b}{L_{2p'}(\N)} \label{augmented}\end{align}
where the infimum runs over all such factorization $x=ayb$. For positive $x\ge 0$, it suffices to consider positive invertible element $a=b\ge 0$ in the infimum. Hence, \begin{align}\norm{\rho}{L_1^{p}(\N\subset\M)}=\inf_{\si\in S(\N) }\norm{\si^{-\frac{1}{2p'}}\rho \si^{-\frac{1}{2p'}}}{p}. \label{positive}\end{align}
where the infimum runs over all density $\si\in L_1(\N)$ such that there exists a factorization $\rho=\si^{\frac{1}{2p'}}y\si^{\frac{1}{2p'}}$ for some positive $y\in L_p(\M)$.
Therefore, for $1< p\le \infty$,
\[D_p(\rho||\N)=p'\log \norm{\rho}{L_1^p(\N\subset \M)} \pl.\]
It follows from H\"older inequality that
 $\norm{x}{L_1^p(\N\subset \M)}\ge \norm{x}{1}$ and $\norm{x}{L_1^p(\N\subset \M)}=\norm{x}{1}$ if and only if $x\in L_1(\N)$.
This corresponds to the property i) in Proposition \ref{basic}.
The dual spaces of amalgamated $L_p$-spaces are conditional $L_p$-spaces. The conditional $L_p$-space $L_\infty^{q}(\N\subset\M)$ is defined as the completion of $L_p(\M)$ with respect to the norm
\[\norm{x}{L_\infty^q(\N\subset\M)}=\sup_{\norm{\pl a\pl }{L_{2q}(\N)}=\norm{\pl b\pl }{L_{2q}(\N)}=1} \norm{axb}{L_q(\M)}\pl.\]
Via the trace pairing $\lan x, y\ran=tr(xy)$, $L_\infty^{p'}(\N\subset\M)\subset L_1^{p}(\N\subset\M)^*$ as $w^*$-dense subspace \cite[Proposition 4.5]{JPmemo}. The connection of $D_p(\rho||\N)$ for $\frac{1}{2}\le p< 1$ goes with conditional $L_p$-norm of $\rho^{\frac{1}{2}}$. Let $1\le q=2p\le 2 $ and $\frac{1}{q}=\frac{1}{r}+\frac{1}{2}$. We define the norm
\[\norm{x}{L_{(r,\infty)}^{2}(\N\subset \M)}=\sup_{\norm{a}{L_r(\N)}=1}\norm{ax}{L_q(\M)}\pl.\]
where the supreme runs over all $a\in \N$ with $\norm{a}{L_{r}(\N)}=1$. For $1\le q=2p<2 $
\[D_p(\rho||\N)=-r\log \norm{\rho^{\frac{1}{2}}}{L^{2}_{(r,\infty)}(\N\subset \M)} \pl.\]
We show that the infimum in $D_p(\rho||\N)$ is always attained. The proof uses uniform convexity of noncommutative $L_p$-spaces and is provided in the appendix.
\begin{prop}\label{unique}For $1/2\le p\le \infty$, the infimum $\displaystyle D_p(\rho||\N)=\inf_{\si\in \S(\N)}D_p(\rho||\si)$ is attained at some $\si$. For $1/2<p<\infty$, such $\si$ is unique.
\end{prop}

\section{Maximal relative entropy}
Recall that the Popa-Pimsner index for a finite von Neumann algebra is defined as
\[ \la(\M:\N)=\max \{\la \pl |\pl  \la x\le E(x)\pl \forall\pl x \in M_+ \}\]
where $\M_+$ is the positive cone of $\M$. This definition can be rewritten via $D_\infty$ as follows
\begin{align*}\log \la(\M:\N)=&\log \sup \{\la \pl |\pl  \la x\le E(x)\pl \text{for all}\pl x \in \M_+ \}
\\ =&\log \inf_{x\in \M_+}\sup \{\la \pl |\pl  \la x\le E(x) \}
\\ =&\inf_{x\in \M_+}(\log \inf \{\mu \pl | \pl  x\le \mu E(x) \})^{-1}
\\ =&\Big(\sup_{x\in \M_+} \log \inf \{\mu \pl |  \pl x\le \mu E(x) \}\Big)^{-1}
\\ =&\Big(\sup_{x\in S(\M)}D_\infty(x||E(x))\Big)^{-1}
\end{align*}
where the last equality follows that $\M_+$ is norm-dense in $L_1(\M)_+$. Thus we have \begin{align}\label{in}-\log \la (\M:\N)=\sup_{\rho\in S(\M)}D_\infty(\rho||E(\rho)).\end{align}
Based on this observation, we prove Theorem A.
\begin{theorem}\label{index}Let $\N\subset \M$ be an inclusion of II$_1$ factors or finite dimensional von Neumann algebras. Then for $1/2\le p\le \infty$,
\begin{align*}
-\log \la (\M:\N)=\sup_{\rho\in S(\M)} D_p(\rho||E(\rho))=\sup_{\rho\in S(\M)} D_p(\rho||\N)
\end{align*}
\end{theorem}
\begin{proof}By monotonicity,
\begin{align*}
&D_{\frac{1}{2}}(\rho||\N)\le D_p(\rho||\N) \le D_\infty(\rho||\N) \le D_\infty(\rho||E(\rho)),\\
&D_{\frac{1}{2}}(\rho||\N)\le D_{\frac{1}{2}}(\rho||E(\rho))\le D_p(\rho||E(\rho))\le D_\infty(\rho||E(\rho))\pl.
\end{align*}
Then it suffices to prove that
\[\sup_{\rho\in S(\M)}D_{\frac{1}{2}}(\rho||\N)\ge -\log \la(\M:\N) \pl.\]
Note that
\begin{align*}D_{\frac{1}{2}}(\rho||\N)&=\inf_{\si\in S(\N)} D_{\frac{1}{2}}(\rho||\si)=\inf_\si -2\log \norm{\si^{\frac{1}{2}}\rho^\frac{1}{2}}{1}\\
&=-2\log\sup_\si \norm{\si^{\frac{1}{2}}\rho^\frac{1}{2}}{1}\pl.
\end{align*}
Let $e=\text{supp}(\rho)$ be the support projection of $\rho$. By H\"older inequality, for any $\si\in S(\N)$,
\begin{align*} \norm{\si^{\frac{1}{2}}\rho^{\frac{1}{2}}}{1}
&\le \norm{\si^{\frac{1}{2}}e}{2} \norm{\rho^{\frac{1}{2}}}{2}
\\&= tr(\si e)^{\frac{1}{2}}= tr(\si E(e))^{\frac{1}{2}}\le \norm{E(e)}{\infty}^{\frac{1}{2}}.
\end{align*}
Therefore, $D_{\frac{1}{2}}(\rho||\N)\ge -\log \norm{E(e)}{\infty}$ and
\begin{align*}
\sup_\rho D_{\frac{1}{2}}(\rho||\N)&\ge -\log \inf\{\norm{E(e)}{} |\pl  e \pl\text{projection in $\M$} \}\pl.
\end{align*}
When $\M,\N$ are II$_1$ factors or finite dimensional, the infimum at the right hand side equals $\la(\M:\N)$ by \cite[Theorem 2.2 \& Corollary 5.6]{pipo}. That completes the proof.
\end{proof}

The above theorem used the monotonicity of $D_p$ over $p$ and the following equality
\begin{align} \label{key} \max \{\la | \la x\le E(x)\pl \forall\pl x \in \M_+ \}= \inf\{\norm{E(e)}{} |\pl  e \pl\text{projection in $\M$} \}\pl. \end{align}
This equality was proved in \cite{pipo} for II$_1$ factors and finite dimensional von Neumann algebras. While "$\le$" direction always holds from convexity, the converse inequality is open in general. In both finite dimensional or subfactor cases, it follows from the fact that there exists a projection $e_0\in \M$ such that $E(e_0)$ is again a projection up to the constant $\la(\M:\N)$. Let $\rho_0=tr(e)^{-1}e$ be the normalized density of $e$. $D_p(\rho_0||\N)$ attains the index for all $1/2\le p\le \infty$,
\begin{align}\sup_{\rho\in S(\M)}D_p(\rho||\N)=D_p(\rho_0||\N)\pl =D_p(\rho_0||E(\rho_0))\pl. \label{optimal}\end{align}
Let us briefly review the value of $\la(\M:\N)$ and the optimal density $\rho_0$ from \cite{pipo}. For II$_1$ factors $\N\subset\M$, there is a projection $e\in \M$ such that $E(e)=[\M:\N]^{-1}1$ (\cite[Lemma 3.1.8]{Jones}). This implies \[\la(\M:\N)^{-1}=[\M:\N]\pl,\]

Let $\N\cong \oplus_{k} M_{n_k}, \M\cong \oplus_{l}M_{m_l}$ be a pair of finite dimensional von Neumann algrebas. Suppose the unital inclusion $\iota: \N \hookrightarrow \M$ is given by
\[\iota(\oplus_{k}x_k)=\oplus_{l}(\oplus_{k}x_k\ten 1_{a_{kl}})\pl.\]
Here $1_{n}$ denotes the identity matrix in $M_n$ and $a_{kl}$ is called the inclusion matrix, which means that each block $M_{m_l}$ of $\M$ contains ${a_{kl}}$ copy of $M_{n_k}$ blocks from $\N$. Let $t_l$ be the trace of minimal projection in $M_{m_l}$ block of $\M$ and $s_k$ be the trace of minimal projection in $M_{n_k}$ block of $\N$. Then $s=(s_k), t=(t_l),n=(n_k),m=(m_l)$ as column vectors satisfy $s=A t$ and $m=A^T n$, where $A=(a_{kl})$ and $A^T$ is the transpose of $A$.

  Based on Theorem \ref{index}, it is equivalent to maximize $D(\rho||E(\rho))=H(E(\rho))-H(\rho)$.
By convexity of $D(\cdot||\N)$, it suffices to consider a minimal projection $e=\ket{\psi}\bra{\psi}\ten 1_{t_l}$ in one block $M_{m_l}$. Then  $\rho=\ket{\psi}\bra{\psi}\ten \frac{1_{t_l}}{t_l}$ is the normalized density and $H(\rho)=\log t_l$. Denote $P_{k,i}$ be the projection in $M_{m_l}$ corresponding to the $i$th copy of $M_{n_k}$ and write $\ket{\psi_{k,i}}=P_{k,i}\ket{\psi}$. The conditional expectation of $\rho$ is given by
\[E(\rho)=\oplus_{k}(\sum_{i=1}^{a_{kl}}\ketbra{\psi_{k,i}})\ten \frac{1}{s_k}1_{s_k}\pl.\]
The largest possible rank of $E(\rho)$ is $\sum_{k}\min (a_{kl},n_k)s_k$ because the part in the $M_{n_k}$ block of $\N$ \[\sum_{i=1}P_{i,k}\ketbra{\psi}P_{i,k}=\sum_{i=1}^{a_{kl}}\ketbra{\psi_{k,i}}\] is of rank at most $\min (a_{kl},n_k)$. Then the maximal entropy $H(E(\rho))$ is attained by choosing $\ketbra{\psi_{k,i}}$ mutually orthogonal and $\norm{\psi_{k,i}}{}^2=\frac{s_k}{\sum_{k}{\min (a_{kl},n_k)}s_k}$. In this case,
\[E(\rho)=\oplus_{k}(\sum_{i=1}^{a_{kl}}\ketbra{\psi_{k,i}})\ten \frac{1}{s_k}1_{s_k}=\frac{1}{\sum_{k}{\min (a_{kl},n_k)}s_k}\oplus_{k}(\sum_{i=1}^{a_{kl}}\ketbra{\tilde{\psi}_{k,i}})\ten 1_{s_k}\]
where $\ket{\tilde{\psi}_{k,i}}=\ket{\psi_{k,i}}/\norm{\psi_{k,i}}{2}$ are unit vectors. Then
\begin{align*}
D(\rho||E(\rho))=&H(E(\rho))-H(\rho)=\log \sum_{k}\min (a_{kl},n_k)s_k-\log t_l
\\=&\log \sum_{k}\min (a_{kl},n_k)s_k/t_l\pl.
\end{align*}
Taking the maximum over the block $M_{m_l}$ of $\M $ leads to the formula in \cite[Theorem 6.1]{pipo},
\begin{align}\label{formula} \max_\rho D(\rho||\N)=-\log \la(\M:\N)=\log \max_l \sum_{k}\min (a_{kl},n_k)s_k/t_l\pl.\end{align}

%For the upper bound, let $e_l$ be the projection of $M_{m_l}$ in $\M$. Note that
%\[e_l\E_\N(e)= \oplus_{k}(\sum_{i=1}^{a_{kl}}\ketbra{\psi_{k,i}})\ten \frac{t_l}{s_k}1_{a_{kl}} \]
%Write $b_k=\frac{1}{s_k}\sum_{i=1}^{a_{kl}}\ketbra{\psi_{k,i}})$. There $\la \E_\N(e)\ge e$ if and only if the matrix
%\[(b_k^{-\frac{1}{2}}\ket{\psi_{k,i}}\bra{\psi_{k',i'}}b_{k'}^{-\frac{1}{2}})_{ki,k'i'}\le t_l\la \pl,\]
%which by a column-row trick is equivalent to
%\[\sum_{i,k}\bra{\psi_{k,i}}b_k^{-1}\ket{\psi_{k,i}}\le t_l\la\pl.\]
%Moreover,
%\[\sum_i\bra{\psi_{k,i}}b_k^{-1}\ket{\psi_{k,i}}=\sum_ks_ktr(b_kb_k^{-1})=\sum_ks_ktr(supp (b_k))\le \sum_k s_k\min(a_{kl},n_k)\pl.\]
%Therefore $\la^{-1}(\M,\N)\le \max_l\sum_k \min(a_{kl},n_k)s_k/t_l$.

Motivated from above we introduce for finite von Neumann algebras $\N\subset\M$, the relative entropy $D(\M||\N)$ and its R\'enyi version $D_p(\M||\N)$
\begin{align*}&D(\M||\N):=\sup_{\rho\in S(\M)} D(\rho||\N)\pl,  \pl D_p(\M||\N):=\sup_{\rho\in S(\M)} D_p(\rho||\N)\end{align*}
As a consequence of Theorem \ref{index}, for II$_1$ subfactors or finite dimensional $\N\subset\M$, $D_p(\M||\N)=D(\M||\N)$ is independent of $p$, while in general such equality is open. These definitions are different with the Connes-Stormer relative entropy
\[ H(\M|\N)=\sup_{\sum_i x_i=1}\sum_{i}tr(x_i\log x_i-x_i\log E(x_i))\]
where the supreme runs over all partition of unity $\sum_{i}x_i=1, x_i\ge 0$. We now discuss the relation between $\la(\M:\N)$, $D_p(\M||\N)$ and $H(\M|\N)$.

\begin{prop}\label{compare}Let $\N\subset \M$ be finite von Neumann algebras.\begin{enumerate}
\item[i)]
$D_p(\M||\N)$ is monotone for $p \in [1/2, \infty]$.
\item[ii)]
For $1\le p\le \infty$, \[-\log\la(\M:\N)\ge D_p(\M||\N)\ge H(\M|\N)\pl.\]
\item[iii)]
If $\N\subset \M$ are II$_1$ subfactors or finite dimensional, then for $\frac{1}{2}\le p\le \infty$,
\[-\log\la(\M:\N)= D_p(\M||\N)\pl.\]
\end{enumerate}
\end{prop}

\begin{proof}i) follows from the monotonicity of $D_p$. For ii), we have by \eqref{in} that
\[-\log\la(\M:\N)= \sup_\rho D_\infty(\rho||\E(\rho))\ge D_\infty(\M||\N)\ge D_p(\M||\N)\pl.\]
Let $x_i\in \M$ such that $\sum_{i=1}^n x_i=1$ and $x_i\ge 0$. Write $\tilde{x}_i=\frac{1}{tr(x_i)}x_i$ as the normalized density. Then
\[ H(\M|\N)=\sup_{\{p_i\},\tilde{x}_i}\sum_{i}p_i D(\tilde{x}_i||\E(\tilde{x}_i))-\sum_{i}p_i\log p_i=\sup_{\{p_i\},\tilde{x}_i}D(\rho||id \ten E_\N(\rho))\]
where $\rho=\sum_{i}p_i \ketbra{i}\ten \tilde{x}_i$ is a density operator in $l_\infty^n(\M)$. Here $l_\infty^n$ is the $n$-dimensional abelian $C^*$-algebra. It follows from convexity that for any finite $n$, $D(l_\infty^n(\M)||l_\infty^n(\N))=D(\M||\N)$.
Then for $1\le p\le \infty$,
\[H(\M|\N)\le \sup_n D(l_\infty^n(\M)||l_\infty^n(\N))=D(\M||\N)\le D_p(\M||\N)\le -\log \la(\M:\N)\pl.\]
iii) is a direct consequence of Theorem \ref{index}.
\end{proof}
\begin{rem}{\rm Recall that Petz's R\'enyi relative entropy is defined as
\[\tilde{D}_p(\rho||\si)=p'\log tr(\rho^p\si^{1-p})^{\frac{1}{p}}\pl.\]
For $p=\frac{1}{2}$, $D_\frac{1}{2}(\rho||\si)\le \tilde{D}_\frac{1}{2}(\rho||\si)$ by definition and for $1<p$,
it was proved in \cite[Corollary 3.3]{Jencova18} that $\tilde{D}_{2-\frac{1}{p}}(\rho||\si)\le D(\rho||\si)\le \tilde{D}_{p}(\rho||\si)$. Therefore, for $\N\subset\M$ II$_1$ factors or finite dimensional, the maximal relative entropy expression also holds for $\tilde{D}_{p}$ with $\frac{1}{2}\le p\le 2$,
\[-\log\la(\M:\N)= \tilde{D}_p(\M||\N):=\sup_{\rho\in S(\M)}\inf_{\si\in S(\N)} \tilde{D}_{p}(\rho||\si)\pl. \]}
\end{rem}

As observed in \cite{pipo}, $-\log \la(\M:\N)=D(\M||\N)$ does not always coincides with $[\M,\N]$ for finite dimensional subfactors. Indeed, for $n<m$, \[D(M_{n}\ten M_m||M_n)=\log \min(n,m)m\neq \log m^2 =\log [M_{n}\ten M_m:M_n]\pl.\] Moreover, the subfactors index satisfies the multiplicative properties \begin{enumerate}
\item[i)] for $\N\subset \M\subset \L$, $[\L:\N]=[\L:\M][\M:\N]$
\item[ii)] for $\N_1\subset \M_1,\N_2\subset \M_2$, $[\M_1\ten\M_2:\N_1\ten\N_2]=[\M_1:\N_1][\M_2:\N_2]$
\end{enumerate}
The follow proposition shows that this also differs with $D(\M||\N)$. Here and in the following, we use the notation $\ten$ for von Neumann algebra tensor product.
\begin{prop}\label{property} Let $\N,\M,\L$ be finite von Neumann algebras.
\begin{enumerate}
\item[i)] for $\N\subset \M\subset \L$, $D(\L||\N)\le D(\L||\M)+D(\M||\N)$;
\item[ii)] for $\N_1\subset \M_1,\N_2\subset \M_2$, $D(\M_1\ten \M_2||\N_1\ten \N_2)\ge D(\M_1||\N_1)+D(\M_2||\N_2)$.
\end{enumerate}
In general both inequalities can be strict.
\end{prop}
\begin{proof} i) Let $E_\M$ (resp. $E_\N$) be the conditional expectation from $\L$ onto $\M$ (resp. $\N$). Because $E_\N\circ E_\M=E_\N$, for $\rho\in S(\L)$,
\begin{align*}D(\rho||\N)&=H(E_\N(\rho))-H(\rho)=H(E_\N(\rho))-H(E_\M(\rho))+H(E_\N(\rho))-H(\rho)
\\ &=D(E_\M(\rho)||\N)+D(\rho||\M)\le D(\M||\N)+D(\L||\M)\end{align*}
which proves i). For the strict inequality case,  we have \begin{align*} D(M_{4}||M_2)=\log 4\pl, \pl D(M_{2}||\C)=\log 2\pl, \pl D(M_{4}||\C)=\log 4\neq D(M_{4}||M_2)+D(M_{2}||\C)\pl.\end{align*}
For ii), let $E_i, i=1,2$ be the conditional expectation from $\M_i$ to $\N_i$. The inequality follows from that
\[D(\rho||E_1(\rho))+D(\si||E_2(\si))=D(\rho\ten \si ||E_1(\rho)\ten E_2(\si))\le D(\M_1\ten \M_2||\N_1\ten \N_2)\pl.\]
This inequality is strict for the case
\begin{align*}
&D(M_{6}||M_2)=\log 6\pl, \pl D(M_{6}||M_3)=\log 4\pl,\\  &D(M_{36}||M_6)=\log 36\neq D(M_{6}||M_2)+D(M_{6}||M_3) \end{align*}
Another example is $\N=({M_2}\ten \C 1_{3}) \oplus ({M_3}\ten \C 1_{2})\subset M_{12}=\M$. Then \begin{align*}&D(M_{12}||\N)=\log(4+6)=\log10\pl, \\&D(M_{12}\ten M_{12}||\N\ten \N)=\log(4\times 9+6\times 6+6\times 6+4\times 4)=\log 126\pl. \qedhere\end{align*}
\end{proof}
\begin{rem}{\rm Form the above example, we know that there exists a bipartite state $\rho\in M_{12}\ten M_{12}$ such that \[D(\rho_1||\N)+D(\rho_2||\N)<D(\rho||\N\ten \N)\pl,\] where $\rho_1$ and $\rho_2$ are the reduced densities of $\rho$ on each component. Hence the relative entropy to subalgebra is {\bf super-additive}. The super-additivity implies that $\rho$ is an entangled state, which means $\rho$ is not a convex combination of tensor product densities. }
\end{rem}
The following is an example of left regular representation of finite groups.
\begin{exam}{\rm Let $G$ be a finite group. Consider the left regular representation $\la:G\to B(l_2(G))$ and the gourp von Neumann algebra $L(G)=\text{span} \la(G)\subset B(l_2(G))$. For a subgroup $H\subset G$, denote $L(H)$ as the subalgebra generated by $\la(H)$. Then for inclusion $L(H)\subset L(G)$,
\[D(L(G)||L(H))=\log[G:H]\pl.\]
To see that, first by Peter-Weyl formula, $L(G)\cong \oplus_{k} M_{n_k}\ten \C 1_{n_k}$ and $|G|=\sum_{k}n_k^2$. We use the formula \eqref{formula},
\[D(L(G)||\C)=\log|G|\pl ,\pl D(B(l_2(G))||L(G))=\log(\sum_{k}n_k^2)=\log|G|.\]
Consider $G=H\cup Hg_{1} \cup \cdots  Hg_{n-1}$ decomposed as a disjoint union of cosets and $n=[G:H]$. Let $P_i$ be the projection onto $l_2(Hg_{i})$ as a subspace of $l_2(G)$. So $L(H)$ is a left regular representation of $H$ of multiplicity $n$ on $\oplus_i P_il_2(G)=l_2(G)$. Thus
\[D(L(H)||\C)=\log|H|\pl, D(L(G)||L(H))\ge \log [G:H]\]
by Proposition \ref{property} i) for the inclusion $\mathbb{C}\subset L(H)\subset L(G)$. On the other hand, the conditional expectation $E_H:L(G)\to L(H)$ is given by
\[E_{H}(\sum_{g\in G}\al_g\la(g))=\sum_{g\in H}\al_g\la(g)=\sum_i{P_{i}(\sum_{g\in G}\al_g\la(g))P_i}\pl,\]
where $\la(g)$ is the unitary of left shifting by $g$. For $g\notin H$, $P_i\la(g)P_i=0$ because for any $h_1,h_2\in H$, $gh_1g_i=h_2g_i$ implies $g=h_2h_1^{-1}\in H$. Note that the trace on $L(G)$ coincides with the induced normalized matrix trace of $B(l_2(G))$. Consider $\N=\oplus B(l_2(Hg_i))\subset B(l_2(G))$. We have $D(B(l_2(G))||\N)=\log n$ and $E_\N(\rho)=\sum_{i}P_i\rho P_i$ is the conditional expectation. Thus
\begin{align*}
D(L(G)||L(H))&=\sup_{\rho\in L(G)}D(\rho||E_H(\rho))= \sup_{\rho\in L(G)}D(\rho||E_\N(\rho)) \\ &\le \sup_{\rho\in B(l_2(G))}D(\rho||E_\N(\rho))=D(\M||\N)=\log n\end{align*}
Therefore we obtain $D(\M||\N)= [G:H]$.}
\end{exam}
The continuity of $D(\cdot || \N)$ follows from that $D(\M||\N)<\infty$, as in \cite[Lemma 7]{winter}
\begin{prop}Suppose $D(\M||\N)<\infty$.
If $\rho$ and $\si$ are two densities of $\M$ such that $\norm{\rho-\si}{1}=\epsilon$, then
\[| D(\rho||\N)-D(\si||\N)| \le 2\epsilon D(\M||\N)+ (1+2\epsilon)h(\frac{2\epsilon}{1+2\epsilon}).
\]
Here $h(\la)=-\la\log \la -(1-\la)\log (1-\la)$ is the binary entropy function.
\end{prop}
Let $l_\infty^n$ be the $n$-dimensional abelian $C^*$-algebra. We know by convexity that tensoring with an commutative space $l_\infty^n$ does not change the relative entropy, \[D_p(l_\infty^n(\M)||l_\infty^n(\N))=D_p(\M||\N)\pl.\]
However this is not the case if we replace $l_\infty$ by a quantum system $M_n$. For finite von Neumann algebras $\N\subset\M$, we define the $cb$-relative entropy that for $1/2\le p\le \infty$
\begin{align*}
&D_{cb,p}(\M||\N):=\sup_{n}D_p(M_n(\M))||M_n(\N))
\end{align*}
In general, $D_{cb,p}(\M||\N)\ge D_p(\M||\N)$ and the inequality can be strict. In particular, for finite dimensional factors
\begin{align}\label{fd}&D_p(M_n\ten M_m||M_n)=mn=-\log \la(M_n\ten M_m: M_n)\pl,\nonumber\\
&D_{p,cb}(M_n\ten M_m||M_n)=m^2=\log [M_n\ten M_m: M_n]\pl.\end{align}
which are different when $n<m$. Using the properties of $D(\M||\N)$, we immediately obtain
\begin{cor}\label{pin}
\begin{enumerate}
\item[i)] $D_{p,cb}(\M||\N)$ is monotone for $p\in [1/2,\infty]$.
\item[ii)] If $\N\subset \M$ are II$_1$ factors or finite dimensional, $D_{p,cb}(\M||\N)$ is independent of $p$.
\item[iii)] For $\N\subset \M$ finite factors, \begin{align}\label{cb} \log[\M:\N]=D_{p,cb}(\M||\N) \pl.\end{align}
\end{enumerate}
\end{cor}
\begin{proof}For iii), the finite dimensional case can be verified use the formula \eqref{fd}. For II$_1$ subfactors, $D_{cb}(\M||\N)=D(\M||\N)=\log [\M:\N]$ because subfactor index $[\M:\N]$ is multiplicative \cite[Proposition 2.1.15]{Jones}.\end{proof}

The above proposition suggests that (the exponential of) $D_{p,cb}$ is a extension of subfactor index $[\M:\N]$ to finite von Neumann algebras. Using the connection between $D_p(\rho||\N)$ and $\norm{\rho}{L_1^p(\N\subset \M)}$ for $1< p\le \infty$, we see that $D_p(\M||\N)$ is basically the norm of identity map from $L_1(\M)$ to $L_1^p(\N\subset \M)$. Indeed, it suffices to consider positive elements because for $x=yz$,
\begin{align*}\norm{x}{L_1^p(\N\subset\M)}\le \norm{y}{L_{2p'}(\N)L_{2p}(\M)}\norm{z}{L_{2p}(\M)L_{2p'}(\N)}\le \norm{yy^*}{L_1^p(\N\subset\M)}\norm{z^*z}{L_1^p(\N\subset\M)}\end{align*}
(see the Appendix for the definition of $L_{2p'}(\N)L_{2p}(\M)$ and $L_{2p}(\M)L_{2p'}(\N)$). Thus, for $1< p\le \infty$,
\begin{align*}
&D_p(\M||\N)=p'\log\norm{id:L_1(\M)\to L_1^p(\N\subset \M)}{}\end{align*}
For $\frac{1}{2}<p<1$ and $\frac{1}{2p}=\frac{1}{r}+\frac{1}{2}$, the relative entropy is
\begin{align*}
&D_p(\M||\N)=2p'\log\norm{id:L_2(\M)\to L^2_{(r,\infty)}(\N\subset \M)}{}\end{align*}
In Appendix Proposition \ref{os4}, we give a natural operator space structure of $L_\infty^{q}(\N\subset\M)$ as follows,
\[M_n(L_\infty^{q}(\N\subset\M))\cong L_\infty^{q}(M_n(\N)\subset M_n(\M))\pl.\]
Based on that, we show $D_{p,cb}$ are indeed given by the completely bounded norms.

\begin{theorem}\label{add}Let $1\le p\le \infty$ and $1/p+{1}/{p'}=1$.
\begin{enumerate}
\item[i)]for $1< p\le \infty$, $D_{p,cb}(\M||\N)=p'\log\norm{id: L_\infty^{p'}(\N \subset \M)\to \M}{cb}$.
\item[ii)] for $1\le p\le \infty$,
\begin{align}D_{p,cb}(\M||\N)=\sup_\R D_p(\R\ten\M||\R\ten\N)\pl.
\label{equal}\end{align}
where the supremum runs over all finite von Neumann algebra $\R$.
\item[iii)] For $\N_i\subset\M_i, i=1,2$ finite von Neumann algebras
    \begin{align}
    D_{p,cb}(\M_1\ten \M_2||\N_1\ten \N_2)\le D_{p,cb}(\M_1||\N_1)+D_{\infty,cb}(\M_2||\N_2) \pl.
    \end{align}
In particular, for $p=1$ and $\infty$, we have the additivity
\begin{align*}
&D_{cb}(\M_1\ten \M_2||\N_1\ten \N_2)=D_{cb}(\M_1||\N_1)+D_{cb}(\M_2||\N_2) \pl.\\
&D_{\infty,cb}(\M_1\ten \M_2||\N_1\ten \N_2)=D_{\infty,cb}(\M_1||\N_1)+D_{\infty,cb}(\M_2||\N_2)\pl.
\end{align*}
\end{enumerate}
\end{theorem}
\begin{proof}Note that $L_\infty^{p'}(\N\subset \M)\subset L_1^p(\N\subset\M)^*$ as $w^*$ dense subspace.
Using duality we have,
\begin{align*}\norm{id:L_1(\M)\to L_1^p(\N\subset\M)}{}=&\norm{id: L_1^p(\N\subset\M)^*\to L_\infty(\M)}{}\\\ge& \norm{id: L_\infty^{p'}(\N\subset \M)\to L_\infty(\M)}{}\pl.\end{align*}
Note that it suffices to consider positive element for $\norm{id:L_1(\M)\to L_1^p(\N\subset\M)}{}$ and $\norm{id: L_\infty^{p'}(\N\subset \M)\to L_\infty(\M)}{}$. Then by Lemma \ref{sup},
\begin{align*}\norm{id:L_1(\M)\to L_1^p(\N\subset\M)}{}=&\sup_{\rho\in \M,\text{density} }\norm{\rho }{L_1^p(\N\subset\M)}
\\=&\sup_{\rho\in \M, \text{density} } \sup_{x\ge 0, \norm{\pl x\pl }{L_\infty^{p'}(\N\subset\M)}\le 1 }tr(\rho x)
\\=&\sup_{x\ge 0, \norm{\pl x\pl }{L_\infty^{p'}(\N\subset\M)} \le 1}\norm{x}{\infty}\\=&\norm{id: L_\infty^{p'}(\N\subset \M)\to L_\infty(\M)}\pl.
\end{align*}
By Proposition \ref{os4},
\begin{align*} &\sup_n\norm{id:L_1(M_n(\M))\to L_1^p(M_n(\N)\subset M_n(\M))}{}\\=&
\sup_n\norm{id: L_\infty^{p'}(M_n(\N)\subset M_n(\M))\to M_n(\M)}{}
\\ =&
\sup_n\norm{id: M_n(L_\infty^{p'}(\N\subset \M))\to M_n(\M)}{}
\\ =&
\sup_n\norm{id:  L_\infty^{p'}(\N\subset \M)\to\M}{cb}\pl.
\end{align*}
Therefore,
\begin{align*}
&D_{p,cb}(\M||\N)=\sup_n D(M_n(\M)||M_n(\N))\\=&p'\log \sup_n\norm{id:L_1(M_n(\M))\to L_1^p(M_n(\N)\subset M_n(\M))}{}
\\ =& p'\log \norm{id:  L_\infty^{p'}(\N\subset \M)\to\M}{cb}
\end{align*}
This proves i). For ii), let $\R\subset B(H)$ and $\rho\in S(\R\ten\M)$ be a normal state on $\R\ten\M$. Let $\phi$ be a normal state on $B(H)\ten\M$ extending $\rho$. Let $\iota:\R\hookrightarrow B(H)$ be the inclusion. $\iota$ is a normal unital completely positive map. Its adjoint on the predual $\iota^\dag: B(H)_*\to R_*$ is the restriction
\[ \iota^\dag(\omega)=\omega|_\R\]
In particular, using the identification $B(H)_*\cong S_1(H)$ and $L_1(\R)=\R_*$, $\iota^\dag:S_1(H)\to L_1(\R)$ is a completely positive trace preserving map and we have
\[\rho=\iota^\dag\ten id_{\M_*}(\phi)\pl.\]
Since $\iota^\dag\ten id_{\M_*}$ send $L_1(B(H)\ten \N)$ to $L_1(\R\ten \N)$, we have by data processing inequality
\[D_p(\rho||\R\ten \N)\le D_p(\phi||B(H)\ten \N)\pl.\]
(Although $B(H)\ten \N\subset B(H)\ten \M$ are semifinite, the space $L_1^p(B(H)\ten \N\subset B(H)\ten\M)$ are defined analogously as in Appendix A.2. There exists a increasing sequence of finite rank projection $(e_n)\subset B(H)$ such that for $\phi_n=(e_n\ten 1)\phi (e_n\ten 1)$, $\lim_{n\to \infty}\norm{\phi_n-\phi}{1}=0$.
Take $\la_n=tr\ten tr_\M(\phi_n)$. We have $\lim_n \la_n=1$ and for the normalized density $\tilde{\phi}_n=\la_n^{-1}\phi_n$,
\[\lim_{n\to \infty}\norm{\tilde{\phi}_n-\phi}{1}=0\pl, \lim_{n\to \infty}\norm{id\ten E(\tilde{\phi}_n)-id\ten E(\phi)}{1}=0\]
For each $n$, $\tilde{\phi}_n$ is also a density of $M_{k_n}(\M)$ for $k_n=\dim (e_nH)$. For $p=1$, by the lower-semicontinuity \cite{Araki76},
\begin{align*} D(\phi|| B(H)\ten \N)=&D(\phi|| id\ten E(\phi))\le \liminf_{n} D(\tilde{\phi}_n|| id\ten E(\tilde{\phi}_n))\\=&D(\tilde{\phi}_n||M_{k_n}(\N))\le D(M_{k_n}(\M)||M_{k_n}(\N))\le D_{cb}(\M||\N)\end{align*}
This proves ii) for $p=1$. For $1<p\le \infty$, we first assume
$D_p(\phi||B(H)\ten \N)$ is finite. Recall the norm expression \[D_p(\phi||B(H)\ten \N)=p'\log\norm{\phi}{L_1^p(B(H)\ten \N\subset B(H)\ten \M)}\pl.\]
where
\[\norm{\phi}{L_1^p(B(H)\ten \N\subset B(H)\ten \M)}=\inf_{\phi=ayb}\norm{a}{L_{2p'}(B(H)\ten \N)}\norm{y}{L_p(B(H)\ten \M)}\norm{b}{L_{2p'}(B(H)\ten \N)}\pl.\]
Given $a,b\in L_{2p'}(B(H)\ten \N)$ and $y\in L_p(B(H)\ten \M)$ such that $\phi=ayb$,
there exists a sequence of projection $(e_n)\subset B(H)$ that
\[\lim_{n}\norm{(e_n\ten 1) a -a}{2p}=\lim_{n}\norm{b(e_n\ten 1) -b}{2p}=0\pl, \lim_{n}\norm{(e_n\ten 1)\phi (e_n\ten 1)-\phi}{1}=0\]
Then
\begin{align*}
&\norm{\phi-(e_n\ten 1)\phi (e_n\ten 1)}{L_1^p(B(H)\ten \N\subset B(H)\ten \M)}\\ \le &\norm{\phi-(e_n\ten 1)\phi }{L_1^p(B(H)\ten \N\subset B(H)\ten \M)}\\&+\norm{(e_n\ten 1)\phi-(e_n\ten 1)\phi (e_n\ten 1)}{L_1^p(B(H)\ten \N\subset B(H)\ten \M)}
\\ \le &\norm{a-(e_n\ten 1)a}{2p'}\norm{y}{p}\norm{b}{2p'}+\norm{a-(e_n\ten 1)a}{2p'}\norm{y}{p}\norm{b(e_n\ten 1)-b}{2p'} \to 0
\end{align*}
Then for the normalized density $\tilde{\phi}_n$ defined as above,
\[\lim_n\norm{\tilde{\phi}_n}{L_1^p(B(H)\ten \N\subset B(H)\ten \M)}=\norm{\phi}{L_1^p(B(H)\ten \N\subset B(H)\ten \M)}\pl.\]
In terms of entropy, we have
\[D_p(\phi||B(H)\ten \N)=\lim_{n}D_p(\tilde{\phi}_n||B(H)\ten \N)=\lim_{n}D_p(\tilde{\phi}_n||M_{n_k}\ten \N)\le D_{p,cb}(\M||\N)\pl.\]
Now consider the case $D_p(\phi||B(H)\ten \N)=+\infty$,
By Lemma \ref{sup}, for any $N>0$ there exists positive $x\in (B(H)\ten\M)_+$ and $\norm{x}{L_\infty^{p'}(B(H)\ten\N\subset B(H)\ten\M)}=1$ such that
\[tr(\phi x)>N\pl.\]
Moreover, since $\phi\in L_1(B(H)\ten \N), x\in (B(H)\ten\M)_+$, there exits a finite rank projection $e\in B(H)$ such that
\begin{align}\label{modify}tr((e\ten 1)\phi(e\ten 1) x)\ge tr(\phi x)-1\ge N-1\pl.\end{align}
Suppose $\dim (eH)=n$.
Then $\tilde{\phi}=tr((e\ten 1)\phi)^{-1}(e\ten 1)\phi(e\ten 1)$ is a density of $M_n(\M)$ and $\tilde{x}=(e\ten 1)x(e\ten 1)\in M_n(\M)$ with $\norm{\tilde{x}}{L_\infty^{p'}(M_n(\N)\subset M_n( \M}\le 1$. Moreover, from \eqref{modify}
\[\norm{\tilde{\phi}}{L_1^p(M_n(\N)\subset M_n( \M))}\ge tr\ten tr_\M(\tilde{\phi} \tilde{x})=tr\ten tr_\M(\tilde{\phi} x)\ge N-1\pl.\]
This means $D_p(\tilde{\phi}|| M_n( \M))\ge p'\log (N-1)$.
Since $N$ can be arbitrary large, we have $D_{p,cb}(\M||\N)=+\infty$.
%Since $B(H)$ is approximate finite dimensional, we can find $\tilde{\rho}_n\in S(M_n(\M)), \tilde{\si}_n\in S(M_n(\N))$ such that
%\[\norm{\tilde{\rho}_n-\tilde{\rho}}{1}\to 0\pl, \norm{\tilde{\si}_n-\tilde{\si}}{1}\to 0\pl.\]
%By the lower-semicontinuity of $D_p$ for $1\le p\le \infty$ \cite[Proposition 3.7]{Jencova18},
%\[ D_p(\tilde{\rho}||\tilde{\si})\le \liminf_{n\to \infty} D_p(\tilde{\rho}_n||\tilde{\si}_n)\le \sup_n D_p(M_n(\M)||M_n(\N))=D_{p,cb}(\M||\N)\pl. \]

For iii), let $E_i:\M_i\to \N_i$ be the conditional expectation. For a density $\rho\in \R\ten\M_1\ten\M_2$,
\begin{align*}
D(\rho|| & id\ten E_1\ten E_2(\rho))=D(\rho|| id\ten id\ten E_2(\rho))+ D(id\ten id\ten E_2(\rho)|| id\ten E_1\ten E_2(\rho))\\&\le D(\R\ten\M_1\ten \M_2)|| \R\ten\M_1\ten \N_2)+D(\R\ten\M_1\ten \N_2)|| \R\ten\N_1\ten \N_2)
\\& \le D_{cb}(\M_1||\N_1)+D_{cb}(\M_2||\N_2)\pl.
\end{align*}
This proves the case $p=1$. For $p>1$, let
$\si \in \N_1
\ten\N_2$ be an invertible density
\[ \norm{\si^{-\frac{1}{2p'}}\rho \si^{-\frac{1}{2p'}}}{p}\le \norm{\si^{-\frac{1}{2p'}}_1\rho \si^{-\frac{1}{2p'}}_1}{p}\norm{\si_{}^{-\frac{1}{2p'}}\si^{\frac{1}{2p'}}_1}{\infty}^2  \]
for some invertible density $\si_1\in \N_1 \overline{
\ten}\M_2$. Consider the analytic family of operator $f(z)=\si^{\frac{-z}{2}}\si_1^{\frac{z}{2}}$. We have
\[ \norm{\si_{}^{-\frac{1}{2p'}}\si^{\frac{1}{2p'}}_1}{\infty}^2\le \norm{\si_{}^{-\frac{1}{2}}\si^{\frac{1}{2}}_1}{\infty}^{\frac{2}{p'}}=\norm{\si_{}^{-\frac{1}{2}}\si_1 \si_{}^{-\frac{1}{2}}}{\infty}^{\frac{1}{p'}}\pl. \]
In terms of relative entropy, we obtain
\[D_{p}(\rho||\si)\le D_{p}(\rho||\si_1)+D_\infty(\si_1||\si)\pl.\]
Taking infimum for both $\si_1\in \N_1\ten \M_2$ and $\si\in \N_1\ten \N_2$, we have
\[D_{p}(\rho||\N_1\ten \N_2)\le D_{p}(\rho||\N_1\ten \M_2)+D_\infty(\si_1||\N_1\ten \N_2)\pl.\]
Taking supremum over $\rho$, we have
\begin{align*}D_p(\M_1\ten \M_2||\N_1\ten \N_2)\le & D_p(\M_1\ten \M_2||\N_1\ten \M_2)+D_\infty(\N_1\ten \M_2||\N_1\ten \N_2)\\\le& D_{p,cb}(\M_1||\N_1)+D_{\infty,cb}(\M_2||\N_2)\pl.\end{align*}
Replacing $\N_1\subset\M_1$ by $\R\ten \N_1\subset\R\ten \M_1$ yields the inequality for $D_{p,cb}(\M_1\ten \M_2||\N_1\ten \N_2)$.
The converse equality follows from choosing tensor product elements.
\end{proof}

\begin{exam}{\rm Let $\N=\oplus_{k}M_{n_k}\ten \mathbb{C}1_{l_k}\subset M_m$ be a subalgebra where $l_k$ is the multiplicity of each block. By the formula \eqref{formula}, we have
 \[\textstyle -\log\la(M_m:\N)=D(M_m||\N)=\log\sum_{k} \min\{l_k,n_k\}l_k\pl, D_{cb}(M_m||\N)=\log\sum_{k} l_k^2 .\]
 In this case, it is clear to see that the $D_{cb}$ is additive but $D$ is not.
 }\end{exam}
Up to this writing, we do not know whether
$D_{cb,p}=D_{cb}$ independent of $p$ holds for general finite von Neumann algebras.
\section{Applications to decoherence time}
In this section, we discuss the applications to decoherence time of quantum Markov process. We start with the continuous time setting. Let $(\M,tr)$ be a finite von Neumann algebra $\M$ equipped with a faithful normal finite trace $tr$. A quantum Markov semigroup $\displaystyle (T_t)_{t\ge 0}:\M\to \M$ is a $w^*$-continuous family of maps that satisfies
\begin{enumerate}
\item[i)] $T_t$ is a normal unital completely positive (normal UCP) map for all $t\ge 0$.
\item[ii)] $T_t\circ T_s=T_{s+t}$ for any $t,s\ge 0$ and $T_0=id$.
\item[iii)] for each $x\in \M$, $t\to T_t(x)$ is continuous in ultra-weak topology.
\end{enumerate}For $\M=B(H)$, quantum Markov semigroups are also called GLKS equations in quantum physics (see \cite{GLKS}). They model
the evolution of open quantum systems which potentially interact with environments. We denote by $A$ the generator of $T_t$, i.e. $A$ is the operator densely defined on $L_2(\M)$ by \[Ax=\text{w$^*$-}\lim_{t\to 0^+} \frac{1}{t}(x-T_t(x))\pl,\]
whose domain is the set of all $x\in \M$ such that the weak$^*$ limit exists. We denote \begin{align}\N=\{ a\in \M \pl |\pl  T_t(a^*)T_t(a)=T_t(a^*a) \pl\text{and}\pl T_t(a)T_t(a^*)=T_t(aa^*) \pl, \forall \pl t\ge 0\}\label{domain}\end{align} as the common multiplicative domain of $T_t$. $\N$ is called \emph{decoherence-free subalgebra}. When $\N=\mathbb{C}1$ is trivial, $T_t$ is called primitive and has a unique invariant state. In general, $(T_t)_{t\ge 0}$ restricted on $\N$ is a semigroup of $*$-homomorphism. We will focus on the case that the semigroup $(T_t)_{t\ge 0}$ is self-adjoint, i.e. for all $x,y\in \M$ and $t\ge 0$,
$tr(x^*T_t(y))=tr(T_t(x)^*y)$. Then $T_t$ is also trace preserving $tr(T_t(\rho))=tr(\rho)$. Moreover $T_t(x)=x$ for all $x\in \N$ because for $a,b\in \N$,
\begin{align*}
tr(aT_{2t}(b))=tr(T_t(a)T_t(b))=tr(T_t(ab))=tr(ab)\pl.
\end{align*}
Let $E:\M\to \N$ be the trace preserving conditional expectation onto $\N$. By the above discussion, we have \[A\circ E =0 \pl, T_t\circ E= E\circ T_t=E \pl.\]

One important functional inequality which relates the convergence of relative entropy is the modified logarithmic Sobolev inequality (MLSI). We say $(T_t)_{t\ge 0}$ satisfies $\la$-modified logarithmic Sobolev inequality (or $\la$-MLSI) for $\la>0$ if for any density $\rho \in \text{dom}A$
\[ \la D(\rho||\N)\le I_A(\rho)=:tr\big((A\rho)\ln \rho\big)\pl,\]
where $I_A(\rho)$ is called the Fisher information or entropy production. This is equivalent to exponential decay of relative entropy \cite{CLSI,bardet}
\begin{align}D(T_t(\rho)||\N)=D(T_t(\rho)||E(\rho))\le e^{-\la t} D(\rho||E(\rho))=e^{-\la t} D(\rho||\N)\label{entropydecay}\pl.\end{align}
Combined with quantum Pinker inequality (c.f. \cite[Theorem 5.38]{watrous}), \[D(\rho||\si)\ge\frac{1}{2}\norm{\rho-\si}{1}^2, \] MLSI gives an estimate of \emph{decoherence time}
\[ t_{deco}(\epsilon)=\min \{t\ge 0\pl | \pl \norm{T_t(\rho)-E(\rho)}{1}\le \epsilon\pl \forall \pl \text{density} \pl \rho \in L_1(\M) \}. \]
Suppose $D(\M||\N)=\sup_\rho D(\rho||\N)<\infty$ is finite, we have
\begin{align}\label{deco} \la-\text{MLSI}\pl\Longrightarrow \pl t_{deco}(\epsilon)\le \frac{1}{\la} \big(2\log\frac{1}{\epsilon}+\log 2D(\M||\N)\big)\pl.\end{align}
Another functional inequality is the spectral gap (also called Poincar\'e inequality).
For $\la>0$, we say $(T_t)_t$ has $\la$-spectral gap (or $\la$-PI) if for any $x\in\M$,
\[ \la\norm{x-E(x)}{2}^2\le tr(x^*Ax)\]
Write $I$ as the identity map on $L_2(\M)$ and $I-E$ is the projection onto the orthgonoal complement $L_2(\N)^{\perp}$.
The spectral gap condition $\la$-PI is that
\begin{align}&\norm{A^{-1}(I-E):L_2(\M)\to L_2(\M)}{}\le \la^{-1} \nonumber\\
\text{or equivalently }
 \pl &\norm{T_t-E:L_2(\M)\to L_2(\M)}{}\le e^{-\la t} \label{2decay}\end{align}
Thus for each $x$, the $L_2$-distance between $T_t(x)$ and its equilibrium $E(x)$ decays exponentially. In general, $\la$-MLSI implies $\la$-PI \cite{bardet}, which means that the entropy decay \eqref{entropydecay} is stronger than $L_2$-norm decay \eqref{2decay}. The next theorem shows that the spectral gap condition also implies an exponential decay bound of relative entropy.
\begin{theorem}\label{d2}
Let $(T_t)_{t\ge 0}:\M\to \M$ be a self-adjoint quantum Markov semigroup and $\N$ be the common multiplicative domain of $T_t$. Suppose $T_t$ satisfies $\la$-PI. Then for density $\rho\in \M$,
\begin{align}D(T_t(\rho)||\N)\le 2e^{-\la t +D_{2}(\rho||\N)/2}\pl.\label{decay}\end{align}
If in additional $D_2(\M||\N)=\sup_\rho D_2(
\rho||\N)<\infty$, then
\[t_{deco}(\epsilon)\le \frac{1}{\la} \big(2\log\frac{2}{\epsilon}+D_2(\M||\N)/2\big)\]
\end{theorem}
\begin{proof}The $\la$-spectral gap property is equivalent to
\[\norm{T_t-E:L_2(\M)\to L_2(\M)}{}\le e^{-\la t}\]
Since both $T$ and $E$ are $\N$-bimodule maps, it follows from \cite[Lemma 3.12]{CLSI} that
\begin{align*}
\norm{T_t-E: L_1^2(\N\subset \M)\to L_1^2(\N\subset \M)}{}&=\norm{T_t-E: L_2^2(\N\subset\M)\to L_2^2(\N\subset\M)}{}\\
&=\norm{T_t-E: L_2(\M)\to L_2(\M)}{}\le e^{-\la t}
\end{align*}
(see Appendix for definition of $L_p^q(\N\subset\M)$ for general $1\le p,q\le \infty$.)
Then for a density $\rho\in \M$,
\begin{align*}D(T_t(\rho)||\N)\le & D_2(T_t(\rho)||\N)\\
\le & 2\log
\norm{T_t(\rho)}{L_1^2(\N\subset \M)}\\
\le & 2\log
\Big ( \norm{E(\rho)}{L_1^2(\N\subset \M)}+\norm{T_t-E(\rho)}{L_1^2(\N\subset \M)}\Big)
\\ \le &2\log(1+e^{-\la t}\norm{\rho}{L_1^2(\N\subset\M)})\le 2e^{-\la t +D_{2}(\rho||\N)/2}\pl.
\end{align*}
Here we used the fact $\norm{E(\rho)}{L_1^2(\N\subset \M)}=\norm{E(\rho)}{1}=1$ because $E(\rho)\in L_1(\N)$. The decoherence time estimate follows from quantum Pinsker inequality.
\end{proof}

Let us compare the above theorem with the decay property \eqref{entropydecay} obtained from $\la$-MLSI. Because the MLSI constant $\le$ PI constant, the asymptotic decay rate in \eqref{decay} is at least as large as the MLSI constant. Nevertheless the constant factor $2e^{ D_{2}(\rho||\N)/2}$ in \eqref{decay} is larger than $D(\rho||\N)$ in MLSI. On the other hand, as mentioned in introduction MLSI of quantum Markov semigroup is not known to be tensor stable. In \cite{CLSI}, we introduced a tensor stable version of LSI called complete logarithmic Sobolev inequality. We say $(T_t)_{t\ge 0}$ satisfies $\la$-complete logarithmic Sobolev inequality (or $\la$-CLSI) if for any $n$, $id_{M_n}\ten T_t: M_n(\M)\to M_n(\M)$ satisfies $\la$-MSLI. It follows from data processing inequality that $\la$-CLSI is stable under tensorization \cite[Section 7.1]{CLSI}.
In particular, CLSI estimates the \emph{complete decoherence time}
defined as follows,
\[t_{c.deco}(\epsilon)=\inf \{t\ge 0 \pl | \pl\norm{id\ten T_t(\rho)-id\ten E(\rho)}{1}\le \epsilon, \pl \forall \pl n\ge 1 \pl\text{and density}\pl \rho\in M_n(\M)\}\]
Suppose $D_{cb}(\M||\N)<\infty$, we have as analog of \eqref{deco}
\[ \la\text{-CLSI}\pl \Longrightarrow t_{c.deco}(\epsilon)\le \frac{1}{\la} \big(2\log\frac{1}{\epsilon}+\log 2D_{cb}(\M||\N)\big)\pl \]
The complete version of decoherence time estimates the convergence rate independent of the dimension of auxiliary system $M_n$.

The complete decoherence time also estimates the loss of entanglement We say a density $\rho\in L_1(M_n(\M))$ is separable if $\rho=\sum_{j}\la_j \omega_j\ten \rho_j$ with $\sum\la_j=1,\la_j\ge 0$ and $\rho_j\in L_1(\M),\omega\in S_1^n$ densities, i.e. $\rho$ is a convex combination of product densities. We define the \emph{$\epsilon$-separablity time} of $T_t$ as follows,
\[t_{sep}(\epsilon)=\inf \{t\ge 0 \pl | \pl \forall \pl n\ge 1 \pl\text{and density}\pl \rho\in M_n(\M)\pl, \inf_{\si, \text{separable}}\norm{id\ten T_t(\rho)-\si}{1}\le \epsilon\}\]
The above definition describes the time $t$ that $T_t(\rho)$ becomes nearly separable. For examples, these $\epsilon$-separable states cannot be used for Bell inequality violations \cite{JP2011,buhrman}. For $\N$ noncommutative, $id\ten E(\rho)$ can contains entanglement and $T_t$ does not have finite $\epsilon$-separablity time.
If $\N$ is commutative, $id\ten T_t(\rho)$ converges to $id\ten E(\rho)$ which is always separable, hence
\[t_{sep}(\epsilon)\le t_{c.deco}(\epsilon)\]
In this case, CLSI also implies $\epsilon$-separablity time. However, it is not clear whether in general $\la$-MLSI implies $\la$-CLSI or complete decoherence time. We refer to \cite{CLSI} for more discussion about CLSI and related examples.

In contrast to MLSI, the spectral gap property (or PI) is stable under tensorization. Indeed, for any $n$, the generator $A$ has the same spectral as $I_{M_n}\ten A$, which is the generator of $id_{{M_n}}\ten T_t$. Based on this, Theorem \ref{d2} also applies to the semigroup $id_{{M_n}}\ten T_t$, which leads to an estimate of complete decoherence time.
\begin{cor}\label{d3}
Let $(T_t)_{t\ge 0}:\M\to \M$ be a self-adjoint quantum Markov semigroup and $\N$ be the decoherence-free subalgebra of $T_t$. Suppose $T_t$ satisfies $\la$-PI. Then for any $n$ and density $\rho\in M_n(\M)$,
\begin{align}D(id\ten T_t(\rho)||M_n(\N))\le 2e^{-\la t +D_{2}(\rho||M_n(\N))/2}\pl.\label{decay}\end{align}
If in additional $D_{2,cb}(\M||\N)<\infty$, then
\begin{align}t_{c.deco}(\epsilon)\le \frac{1}{\la} \big(2\log\frac{2}{\epsilon}+D_{2,cb}(\M||\N)/2\big)\label{est}\end{align}
\end{cor}
\begin{rem}{\rm The above theorem applies for all finite dimensional self-adjoint semigroup because the generator $A$ always has positive spectral gap. In particular, we obtain that for all self-adjoint semigroup $T_t$ whose fixpoint subalgebra $\N$ is commutative, $T_t$ admits an $\epsilon$-separablity time that is independent of dimension of entangled system}
\end{rem}

\begin{rem}{\rm
The above theorem also applies for tensor product of semigroups. Indeed,
 for two semigroups  $S_t:\M_1\to \M_1$ and $T_t:\M_2\to \M_2$, \begin{enumerate}
\item[i)] If $S_t$ satisfies $\la_1$-PI and $T_t$ satisfies $\la_2$-PI, then $S_t\ten T_t$ satisfies $\min \{\la_1,\la_2\}$-PI.
\item[ii)] If $D_{2,cb}(\M_1||\N_1)<\infty$ and $D_{\infty,cb}(\M_2||\N_2)<\infty$, then $D_{2,cb}(\M_1\ten \M_2||\N_1\ten \N_2)\le D_{2,cb}(\M_1||\N_1)+D_{\infty,cb}(\M_2||\N_2)<\infty$
    by Theorem \ref{add}. Moreover, we know that for finite dimensional $\M_1$ and $\M_2$, \[D_{2,cb}(\M_1\ten \M_2||\N_1\ten \N_2)=D_{cb}(\M_1\ten \M_2||\N_1\ten \N_2)=D_{cb}(\M_1||\N_1)+D_{cb}(\M_2||\N_2)\pl.\]
    \end{enumerate}}
    \end{rem}
We discuss the generalized dephasing map as examples. Let $a=(a_{ij})_{i,j=1}^{m}\in M_m$. The Schur multiplier of $a$ is defined as
\[T_a(x_{ij})=(a_{ij}x_{ij})\]
It is known \cite[Theorem 3.7]{paulsen} that $T_a$ is completely positive if and only if $a\ge 0$; is unital (or equivalently trace preserving) if and only if $a_{ii}=1$; is self-adjoint if and only if $a_{ij}=a_{ji}$.
\begin{exam}{\rm Let $T_t((x_{i,j}))=(e^{-b_{ij}t}x_{ij})$ be a semigroup of Schur multiplier. The generator of $T_t$ is the Schur multiplier of $b=(b_{ij})$,
\[A((x_{i,j}))=(b_{ij}x_{ij})\pl.\]
By Schoenberg's theorem \cite{Schoenberg38}, $T_t$ are unital completely positive trace preserving and self-adjoint if and only if $b_{ii}=0, b_{ij}=b_{ji}\ge 0$ and \emph{conditional negative definite}, i.e. for any real sequence $(c_1,\cdots, c_m)$ with $\sum_{i=1}^mc_i=0$,
\[\sum_{i,j=1}^mc_ic_jb_{i,j}\le 0\pl.\]
For $T_t$, the subalgebra $\N$ is
\[\N=\{\sum x_{ij}e_{ij}\pl | \pl x_{ij}=0 \pl \text{for all $(i,j)$ that $b_{i,j}=0$} \}\pl.\]Lnd $e_{ij}\in M_m$ be the unit matrices in $M_m$.
Because $e_{ij}$ are eigenvector of the generator $A$ with eigenvalue $b_{ij}$, then the spectral gap is
\[\la=\min \{b_{i,j}\pl  |\pl b_{i,j}\neq 0\pl \}\pl.\]
Let us assume that \begin{align}\la=\min_{i\neq j}|a_{ij}|>0\pl.\ \label{assume1}\end{align}
Then $\N\cong l_\infty^m$ is the commutative subalgebra of diagonal matrices and the conditional expectation $E(\sum x_{ij}e_{ij})=\sum x_{ii}e_{ii}$ is the completely dephasing channel. $D(\rho||\N)$ is exactly the relative entropy of coherence in \cite{winter16}, and by Corollary \ref{pin} ii) and formula \eqref{formula},
\[D_{2,cb}(M_m||l_\infty^m)=D_{cb}(M_m||l_\infty^m)=\sup_nD(M_n(M_m)||M_n(l_\infty^m))= m \pl.\]
Then Corollary \ref{d3} implies that 
\[t_{sep}(\epsilon)\le t_{c.deco}(\epsilon)\le \frac{1}{\la} \big(2\log\frac{2}{\epsilon}+m/2\big)\pl.\]
}\end{exam}

We now discuss the discrete time setting. A quantum Makrov map $\displaystyle T:M\to M$ is a normal completely positive unital map. \[\N=\{a\in \M \pl | \pl T(a^*a)=T(a^*)T(a)\pl \text{and}\pl  T(aa^*)=T(a)T(a^*) \}\] be the multiplicative domain of $T$. $T$ restricted on $\N$ is a normal trace preserving $*$-homomorphism. Suppose $T$ is self-adjoint with respect to trace
$tr(xT(y))=tr(T(x)y)$. Then $T^2$ is identity on $\N$ because for any $a,b\in \N$
\[tr(aT^2(b))=tr(T(a)T(b))=tr(T(ab))=tr(ab)\pl.\]
and hence $T$ is a isometry on $L_2(\N)$. For the conditional expectation $E:\M\to \N$, we have
\begin{align} T^2\circ E=E\circ T^2= E\pl, T\circ E=E\circ T \pl.\label{relation}\end{align}
\begin{theorem}\label{dis}
Let $T:\M\to \M$ be a self-adjoint quantum Markov map and let $\N$ be multiplicative domain of $T$. Suppose $\norm{T(I-E):L_2(\M)\to L_2(\M)}{}\le \mu <1$.
Then for any $n\ge 1$ and density $\rho\in M_n(\M)$, we have
\[D(T^{k}(\rho)||M_n(\N))\le 2 \mu^{k} e^{D_{2}(\rho||M_n(\N))/2}\pl.\]
Moreover, for $k\ge (\log\frac{1}{\mu})^{-1}(\log (4/\epsilon^2) +D_{2,cb}(\M||\N)/2)$,
\begin{align*} \norm{id\ten T^k(\rho)-id \ten E (\rho)}{1}\le \epsilon \pl \pl \pl \text{for $k$ even,}\\
\norm{id\ten T^k(\rho)-id \ten T\circ E (\rho)}{1}\le \epsilon \pl \pl \pl \text{for $k$ odd.}\\
\end{align*}
\end{theorem}
\begin{proof}Using the relation \eqref{relation}, we have
\[(T(I-E))^2= (T-T\circ E)^2= T^2-2T^2\circ E+ T^2\circ E=T^2-E\pl.\]
Then
\[ (T-T\circ E)^{2k}=T^{2n}-E \pl, \pl (T-T\circ E)^{2k+1}=T^{2k+1}-E\circ T\pl.\]
By \cite[Lemma 3.12]{CLSI} again, since $(T-E)^k$ are $\N$-bimodule map,
\begin{align*}&\norm{(T-T\circ E)^k:L_1^2(\N\subset\M)\to L_1^2(\N\subset\M)}{}\\=&\norm{(T-T\circ E)^k:L_2(\M)\to L_2(\M)}{}\le \mu^k\end{align*}
The rest of argument is similar to Theorem \ref{d2}. Here we show the case for $k$ odd,
\begin{align*}D( id\ten T^{k}(\rho)||\N)\le & D_2(I\ten T^{k}(\rho)||\N)\\
\le & 2\log
\norm{T^{k}(\rho)}{L_1^2(\N\subset \M)}\\
\le & 2\log
\Big ( \norm{T\circ E(\rho)}{L_1^2(\N\subset \M)}+\norm{(T-T\circ E)^{k}(\rho)}{L_1^2(\N\subset \M)}\Big)\\
\le & 2\log
( 1+\mu^k\norm{\rho}{L_1^2(\N\subset \M)})
\\ \le &2\mu^k e^{D_{2}(\rho||\N)/2}\pl.
\end{align*}
Applying the same argument for $\rho\in M_n(\M)$ yields the desired estimate.
\end{proof}
We end the discussion with Markov map as a Schur multiplier.
\begin{exam}{\rm
 Let $a=(a_{ij})_{i,j=1}^{m}\in M_m$. The Schur multiplier
\[T_a(x_{ij})=(a_{ij}x_{ij})\]is a quantum Markov map if and only if $a$ is a real symmetric positive matrix with $a_{ii}=1$. Then multiplicative domain of $T_a$ is
\[\N=\{\sum x_{ij}e_{ij}\pl | \pl x_{ij}=0 \pl \text{for all $(i,j)$ that $|a_{i,j}|<1$} \}\pl.\]
Let us assume that \begin{align}\mu=\max_{i\neq j}|a_{ij}|<1\pl.\ \label{assume}\end{align}
Then $\N\cong l_\infty^m$ is the diagonal matrices in $M_m$.
Because $e_{ij}$ are eigenvector of $T_a$ with eigenvalue $a_{ij}$, the spectral gap is
\[\norm{T_a(I-E):L_2(M_m)\to L_2(M_m)}{}=\mu <1\pl.\]
Therefore, by Theorem \ref{dis}, for any $n\ge 1$ and density $\rho\in M_n(M_m)$, we have \[\norm{id\ten T_a^k(\rho)-id \ten E (\rho)}{1}\le \epsilon\] whenever
\begin{align}\label{ent}k\ge (\log\frac{1}{\mu})^{-1}(\log (4/\epsilon^2) +m/2)\pl.\end{align}}
\end{exam}

\appendix
\section{}
\subsection{Amalgamated $L_p$-space and Conditional $L_p$-spaces}
In this section, we recall the definition of amalgamated $L_p$-space and conditional $L_p$-spaces for semifinite von Neumann algebras. For the case of general von Neumann algebras, we refer to \cite{JPmemo}. Let $\M$ be a semifinite von Neumann algebra equipped with a normal semifinite faithful trace $tr$. Let $\N\subset \M$ be a subalgebra such that $tr|_{\N}$ is also semifinite. For $1\le p\le q\le \infty$ and ${1}/{p}-{1}/{q}={1}/{r}$, we define the amalgamated $L_p$-space $L_p^q(\N\subset\M)$ as the set of all $x\in L_p(\M)$ which admits a factorization $x=ayb$ with $a,b\in L_{2r}(\N), y\in L_q(\M)$ equipped with the norm
\[\norm{x}{L_p^q(\N\subset\M)}
                             =\inf_{x=ayb, \pl a,b\in\N} \norm{a}{L_{2r}(\N)}\norm{y}{L_q(\M)}\norm{b}{L_{2r}(\N)}\pl.\]
For $1\le q\le p\le \infty$ and ${1}/{q}-{1}/{p}={1}/{r}$, the conditional $L_p$-space $L_p^q(\N\subset\M)$ is the completetion of $L_p(\M)$ with respect to the norm
\[\norm{x}{L_p^q(\N\subset\M)}=\sup_{\norm{\pl a\pl }{L_{2r}(\N)}=\norm{\pl b\pl }{L_{2r}(\N)}=1} \norm{axb}{L_q(\M)}\pl.\]
It follows from H\"older inequality that
\begin{enumerate}
\item[i)] $L_p^p(\N\subset\M)=L_p(\M)$,
\item[ii)] for $q_1\le p \le q_2$,
$\norm{x}{L_p^{q_1}(\N\subset \M)}\le \norm{x}{L_p(\M)}\le \norm{x}{L_p^{q_2}(\N\subset \M)}\pl,$
\item[iii)] $L_p(\N)\subset L_p^q(\N\subset\M)$ for any $1\le q\le \infty$. Moreover, $\norm{x}{L_p^q(\N\subset\M)}=\norm{x}{L_p(\N)}$ if and only if $x\in L_p(\N)$
\end{enumerate}
For $1<p,q<\infty,\frac{1}{p}+\frac{1}{p'}=1$ and $\frac{1}{q}+\frac{1}{q'}=1$, we have the duality
$L_q^p(\N\subset\M)^*=L_{q'}^{p'}(\N\subset\M)$ via
\[\norm{x}{L_q^p(\N\subset\M)}=\sup \{|tr(xy)|\pl | \norm{y}{L_{q'}^{p'}(\N\subset\M)}\le 1\}\pl,\]
For $q=1$, $L_1^p(\N\subset\M)\subset L_\infty^{p'}(\N\subset\M)^*$ as a $w^*$-dense subspace. (see \cite[Propsition 4.5]{JPmemo}).
 The complex interpolation relation is also proved in \cite{JPmemo},
\[ L_q^p(\N\subset\M)=[L_{q_0}^{p_0}(\N\subset\M),L_{q_1}^{p_1}(\N\subset\M)]_\theta\]
isometrically where $(1-\theta)/p_0+\theta/p_1=1/p, (1-\theta)/q_0+\theta/q_1=1/q$ and $(p_1-q_1)(p_2-q_2)\ge 0$.

We will also need some asymmetric version of above $L_p$-spaces. For $2\le r\le \infty, 1\le p,q\le \infty $ and $\frac{1}{q}=\frac{1}{r}+\frac{1}{p}$, we define the norm
\[\norm{x}{L_{(r,\infty)}^{p}(\N\subset \M)}=\sup_{\norm{\pl a\pl}{L_r(\N)}=1}\norm{ax}{L_q(\M)}\pl.\]
where the supreme runs over all $a\in L_{r}(\N)$ with $\norm{a}{L_{r}(\N)}=1$. The dual spaces are the amalgamated space $L_{q'}(\M)L_r(\N)$ given by
\[\norm{y}{L_{q'}(\M)L_r(\N)}=\inf_{y=za} \norm{z}{L_{q'}(\M)}\norm{a}{L_r(\N)}\pl.\]
For $1< q<\infty$, we have the dual relation
\begin{align}\norm{x}{L_{(r,\infty)}^{2}(\N\subset \M)}&=\sup\{\norm{ax}{L_q(\M)} \pl | \pl \norm{a}{L_r(\N)}=1\}\nonumber\\
&=\sup\{|tr(zax)| \pl | \pl \norm{a}{L_r(\N)}=1, \norm{z}{L_{q'}(\M)}=1\}\nonumber\\
&=\sup \{|tr(yx)| \pl | \norm{y}{L_{q'}(\M)L_r(\N)}=1\}\label{dual}
 \end{align}
These spaces also interpolates (see Theorem 4.6 from \cite{JPmemo}).
Note that the property ii) and iii) in Proposition \ref{basic} can also be obtained from complex interpolation relation of the space $L_q^p(\N\subset\M)$ and $L_{(r,\infty)}^{p}$ proved in \cite{JPmemo}. We now prove Proposition \ref{unique}.
\begin{prop}For $1/2\le p\le \infty$, $\displaystyle D_p(\rho||\N)=\inf_{\si\in \S(\N)}D_p(\rho||\si)$ attains the infimum at some $\si$. For $1/2<p<\infty$, such $\si$ is unique.
\end{prop}
\begin{proof}The case for $p=1$ follows from \eqref{p}. For $1<p<\infty$, we use the norm expression
\[D_p(\rho||\N)=p'\log \inf_{\rho=aya}\norm{a}{2p'}^2\norm{y}{p}=\inf_{\rho^{\frac{1}{2}}=a\eta}\norm{a}{2p'}^2\norm{\eta}{2p}^2\pl,\]
where $a\in L_{2p'}(\N), y\in L_p(\M),\eta\in L_{2p}(\M)$ and $a\ge 0$ positive. It suffices to show that the above infimum is attained at unique $a$. Assume $\norm{x}{L_1^p(\N\subset\M)}=1$. We find sequences $(a_n)\subset L_{2p'}(\N)$ and $(\eta_n)\subset L_{2p}(\M)$ such that for each $n$, $\sqrt{x}=a_n \eta_n$, $\norm{a_n}{2p'}=1$ and \[\norm{\eta_n}{2p}\ge 1 \pl ,\pl \lim_{n\to \infty}\norm{\eta_n}{{2p}} \to 1\pl.\]
By replacing $a_n$ with invertible element $\norm{a_n+\delta1}{2p}^{-1}(a_n+\delta1)$, we can assume that each $a_n\ge \delta_n 1$ for some $\delta_n>0$.
Write $a_{n,m}=(\frac{1}{2}a_n^2+\frac{1}{2}a_m^2)^{\frac{1}{2}}$.
Consider the factorization
	\[ \sqrt{x}=\left[\begin{array}{cc}\frac{a_n}{\sqrt{2}} & \frac{a_m}{\sqrt{2}}\end{array}\right]\cdot\left[\begin{array}{c}\frac{\eta_n}{\sqrt2} \\ \frac{\eta_m}{\sqrt2}\end{array}\right]=a_{n,m} \eta_{n,m}
\pl,\]
where $\eta_{n,m}=a_{n,m}^{-1}(\frac{1}{2}a_n\eta_n+\frac{1}{2}a_m\eta_m)$.
Note that
	\begin{align*}
	\norm{a_{n,m}}{2p'}&= \left\| \frac{a_n^2+a_m^2}{2}\right\|_{p'}^{\frac{1}{2}}\pl,\\
	\norm{\eta_{n,m}}{2p}&= \norm{\left[\begin{array}{cc}\frac{a_{n,m}^{-1}a_n}{\sqrt{2}} & \frac{a_{n,m}^{-1}a_m}{\sqrt{2}}\end{array}\right]\cdot\left[\begin{array}{c}\frac{\eta_n}{\sqrt2} \\ \frac{\eta_m}{\sqrt2}\end{array}\right]}{2p}\le
\left\| \frac{\eta_n^*\eta_n+\eta_m^*\eta_m}{2} \right\|_{p}^{\frac{1}{2}}\\ &\le (\frac{1}{2}\norm{\eta_n}{2p}^2+\frac{1}{2}\norm{\eta_m}{2p}^2)^\frac{1}{2}
	\end{align*}
which converges to $1$ when $n,m\to \infty$. Because $\sqrt{x}=a_{n,m}\eta_{n,m}$, we have $\norm{a_{n,m}}{2p'}\norm{\eta_{n,m}}{}
\ge 1$ for any $n,m$. Then we have\[\lim_{N\to \infty} \inf_{n,m\ge N}\norm{\frac{a_n^2+a_m^2}{2}}{p'} \ge 1\pl.\] By uniform convexity of noncommutative $L_{p'}$ spaces (c.f. \cite{kosaki84,fack86}), this implies that $(a_n^2)$ converges in $L_{2p'}$. Using the inequality $\norm{a^2-b^2}{2p}\ge \norm{a-b}{p}^{\frac{1}{2}}$ from \cite[Lemma 1.2]{caspers}, we have that $(a_n)$ converges in $L_{p'}(\N)$. On the other hand, because $L_{2p}(\M)$ is a dual space, there exists a subsequence $\eta_{n_k}\to \eta$ weakly and $\norm{\eta}{2p}\le 1$. Thus $\sqrt{x}=a_{n_k}\eta_{n_k}\to a\eta$ weakly in $L_2(\M)$. Hence $\sqrt{x}=a\eta$ and $\norm{a}{2p'}=\norm{\eta}{2p}=1$. Note that we have shown that for any sequence $a_n$ with $\sqrt{x}=a_{n}\eta_{n}$ and
\begin{align}\norm{a_n}{2p'}=1, \lim_{n\to \infty}\norm{\eta_n}{{2p}} \to 1, \label{con}\end{align}
 $a_n$ converges to some $a$ in $L_{2p'}$. Let $b_n$ be another such sequence with $\sqrt{x}=b_n\eta_n'$ and converges to $b$. Define $c_{2n-1}=a_n, c_{2n}={b_n},\xi_{2n-1}=\eta_n,\xi_{2n}=\eta'_{n}$. Then $\sqrt{x}=c_n\xi_{n}$ satisfies same condition of \eqref{con}. Then $c_n$ converges to some $c$ in $L_{2p'}$ which implies that the limit $a=b=c$ is unique. For $p=\infty$, we know
 \[D_\infty(\rho||\N)=\log \inf \{\la \pl  | \rho\le \la\si , \text{for some density} \si \in L_1(\N) \}\pl.\]
 Let $\la=\inf \{\la \pl  | \rho\le \la\si , \si\in S(\N)\}$ and let $\si_n$ be a sequence of densities in $L_1(\N)\cong \N_*$ such that $\la_n:=\min \{\la \pl  | \rho\le \la\si_n\}\to \la $ monotonically non-increasing. By $w^*$-compactness of state space in $\N^*$, we have a subsequence $\si_{n_k}$ converges to some state $\si\in \N^*$ in the weak$^*$ topology.
 Then for any $k$, $\la_{n_k}\si_{n_m}\ge \rho$ in $\N^*$ for $m\ge k$. Passing to the limit, we have $\la \si\ge \rho$ for some state $\si\in \N^*$. We show that $\si\in \N_*$. By the decomposition of the dual space $\N^{*}=\N_*\oplus \N_*^\perp$, $\si=\si_n\oplus \si_s$ decomposex as a normal part $\si_n\in \N_*$ and a singular part $\si_s\in \N_*^\perp$. Suppose $\si_s \neq 0$. Then $\si_0(1)=\mu <1$ and
 \[\rho \le \la\si \Rightarrow \rho \le\la\si_0\pl.\]
 Take the normalized density $\tilde{\si}=\frac{1}{\mu}\si_0\in \N_*$. We have $ \displaystyle \rho \le \frac{\la}{\mu}\tilde{\si} $
 with $\la/\mu >\la$ which is a contradiction. This proves the existence of $\si\in \N_*\cong L_1(\N)$.

 For $1< q=2p<2$ and $\frac{1}{q}=\frac{1}{r}+\frac{1}{2}$., it sufficient to show that the norm \[\norm{\rho^{\frac{1}{2}}}{L_{(r,\infty)}^{2}(\N\subset \M)}=\sup_{\norm{a}{L_r(\N)}=1}\norm{a\rho^{\frac{1}{2}}}{L_q(\M)}\]
 is attained for some $\norm{a}{L_r(\N)}=1$. Let $\norm{\rho^{\frac{1}{2}}}{L_{(r,\infty)}^{2}(\N\subset \M)}=\la$ and $a_n\ge 0$ be a positive sequence in $\norm{a_n}{L_r(\N)}=1$ such that $\norm{a_n\rho^{\frac{1}{2}}}{L_q(\M)}\to \la$. Write $a_{n,m}=(\frac{a_n^2+a_m^2}{2})^{\frac{1}{2}}$. We have
\begin{align*}
\left[\begin{array}{cc}a_n\rho^{\frac{1}{2}}&a_n\rho^{\frac{1}{2}} \\ a_m \rho^{\frac{1}{2}}&a_m\rho^{\frac{1}{2}}\end{array}\right]
= \left[\begin{array}{cc}a_na_{n,m}^{-1}&0 \\ a_ma_{n,m}^{-1}&0\end{array}\right] \cdot \left[\begin{array}{cc}a_{n,m} \rho^{\frac{1}{2}}&a_{n,m} \rho^{\frac{1}{2}} \\ 0&0\end{array}\right]
\end{align*}
For $n,m$ large enough $\norm{a_n\rho^{\frac{1}{2}}}{L_q(\M)}, \norm{a_m\rho^{\frac{1}{2}}}{L_q(\M)}\ge  (1-\epsilon)\la$. Then we have
\begin{align*}
&\norm{\left[\begin{array}{cc}a_n\rho^{\frac{1}{2}}& a_n\rho^{\frac{1}{2}} \\ a_m\rho^{\frac{1}{2}} &a_m\rho^{\frac{1}{2}}\end{array}\right]}{L_q(M_2(\M))}\ge \norm{\left[\begin{array}{cc}a_n\rho^{\frac{1}{2}}&\ \\ &a_m\rho^{\frac{1}{2}}\end{array}\right]}{L_q(M_2(\M))}\ge 2^{\frac{1}{q}}(1-\epsilon)\la\pl, \\ &\norm{\left[\begin{array}{cc}a_na_{n,m}^{-1}&0 \\ a_ma_{n,m}^{-1} &0\end{array}\right]}{L_\infty(M_2(\M))}=1\\ &\norm{\left[\begin{array}{cc}a_{n,m} \rho^{\frac{1}{2}}&a_{n,m} \rho^{\frac{1}{2}} \\ 0&0\end{array}\right]}{L_q(M_2(\M))}=\norm{\left[\begin{array}{cc}1&1 \\ 0&0\end{array}\right]}{L_q(M_2)}\norm{a_{n,m} \rho^{\frac{1}{2}}}{L_q(\M)}=2^{\frac{1}{q}}\norm{a_{n,m} \rho^{\frac{1}{2}}}{L_q(\M)}
\end{align*}
By the definition of $\la$,
\[(1-\epsilon)\la\le \norm{a_{n,m}\rho^{\frac{1}{2}}}{L_q(\M)} \Rightarrow (1-\epsilon)\le\norm{a_{n,m}}{L_r(\N)} \pl.\]
Thus we have shown \[\lim_{N\to \infty} \inf_{n,m\ge N}\norm{\frac{a_n^2+a_m^2}{2}}{\frac{r}{2}} \ge 1\pl.\]
Following the same argument of the case of $1<p<\infty$,
we obtain that $a_n$ converges $a$ in norm of $L_{r}(\N)$ with $\norm{a\rho^{\frac{1}{2}}}{q}=\la$, and such limit $a$ is unique for $\rho^{\frac{1}{2}}$.
Finally, we discuss the case for $p=1/2$. It suffices to show the following supremum is attained \begin{align}\label{b}\norm{z}{L_{(2,\infty)}^{2}(\N\subset \M)}&=\sup \{ \norm{az}{L_1(\M)}| \norm{a}{L_2(\N)}=1\}\nonumber\\
&=\sup \{ |tr(azy)| | {\norm{a}{L_2(\N)}=1},y\in \M \pl\text{unitary}\}\nonumber\\
&=\sup \{ \norm{E(zy)}{2} | y\in \M \pl\text{unitary}\}\pl.
\end{align}
Consider the set \[C=\{(id-E)(zy) \pl |\pl  y\in \M \pl\text{unitary}\}\pl.\]
$C$ is a weakly convex closed set in $L_2(\M)$. Indeed, for any net $y_\al$ such that $(id-E)(zy_\al)\to x$ weakly in $L_2(\M)$, we can find a subnet $y_{\beta}\to y$ weakly in $\M$. Then $(id-E)(zy_{\beta})\to (id-E)(zy)$ weakly in $L_2(\M)$. Hence $x=(id-E)(zy)$ which proves the closeness. We show that $C$ admits an element attains the infimum
\[\inf_{x\in C}\norm{x}{L_2(\M)}:=\la\]
Let $x_n$ be a sequence such that $\norm{x_n}{2}\to \la$. For a weakly converging subsequence $x_{n_k}\to x$, we have $x\in C$ by closeness and
\[\norm{x}{2}\le \liminf_{k\to \infty} \norm{x_{n_k}}{2}=\la\pl.\]
Hence the infimum norm for is attained.
 %$id-E(zy_1)\neq id-E(zy_2)$ and
%\[\norm{id-E(zy_1)}{2}=\norm{id-E(zy_1)}{2}=\la\]
%Then $\norm{id-E(z\frac{y_1+y_2}{2})}{2}< \la$ with $\norm{\frac{y_1+y_2}{2}}{\infty}\le 1$. Then by Russo–Dye theorem,
Since $E:L_2(\M)\to L_2(\N)$ is a projection,
\[\norm{E(zy)}{2}^2+\norm{(id-E)(zy)}{2}^2=\norm{zy}{2}^2=1\]
We have the supremum \[\sup\{ \norm{E(zy)}{2}\pl |\pl y\in\M \pl \text{unitary}  \pl \}\] is attained by some $y_0$. Therefore the supremum in \eqref{b} is attained with $a=|E(zy_0)|$.
\end{proof}

\subsection{Operator space structures}
We shall now discuss the operator space structures of $L_1^p(\N\subset\M)$. Recall that $L_\infty^{p'}(\N\subset\M)\subset L_1^p(\N\subset\M)^*$ as a weak$^*$-dense subspace. We first consider the operator space structure on $L_\infty^{p'}(\N\subset\M)$ and induce the structure for $L_1^p(\N\subset\M)$ via duality. For $p=\infty$ and $p'=1$, the norm of $L_\infty^1(\N\subset\M)$ is given by
\begin{align*}\norm{x}{L_\infty^1(\N\subset\M)}=\sup_{\norm{\pl a\pl }{L_2(\N)}=\norm{\pl b\pl }{L_2(\N)}=1}\norm{axb}
{L_1(\M)}\pl.\end{align*}
Define its operator space structure as follows,
\begin{align*}&M_{n}(L_\infty^1(\N\subset\M))\cong L_\infty^1(M_n(\N)\subset M_n(\M))\pl.
\end{align*}
We verify the above norms satisfies Ruan's axioms (c.f. \cite{Ruan}). Denote $e_1\in M_{n+m}(\M)$ be the projection for $M_n(\M)$ and $e_2$ for $M_m(\M)$.
Consider $x=x_1\oplus x_2=e_1xe_1+e_2xe_2\in M_n(\M)\oplus M_m(\M)$. Then
\begin{align*}&\norm{x}{L_\infty^1(M_{n+m}(\N)\subset M_{n+m}(\M))}\\=&\sup_{\norm{\pl a\pl }{2}=\norm{\pl b\pl }{2}=1}\norm{axb}{L_1(M_{n+m}(\M))}
\\=&\sup_{\norm{\pl a\pl }{2}=\norm{\pl b\pl }{2}=1}\norm{ax_1b+ax_2b}{L_1(M_{n+m}(\M))}
\\ \le &\sup_{\norm{\pl a\pl }{2}=\norm{\pl b\pl }{2}=1}\norm{ae_1x_1 e_1b}{L_1(M_{n+m}(\M))}+\norm{ae_2x_2e_2b}{L_1(M_{n+m}(\M))}\pl.
\\ \le &\sup_{\norm{\pl a\pl }{2}=\norm{\pl b\pl }{2}=1}\norm{|ae_1|x_1 |(e_1b)^*|}{L_1(M_{n}(\M))}+\norm{|ae_2|x_2|(e_2b)^*|}{L_1(M_{m}(\M))}\pl.
\end{align*}
For $a,b\in L_2(M_{n+m}(\M))$,
\begin{align*}&\norm{a}{2}^2=\norm{ae_1}{2}^2+\norm{ae_2}{2}^2=\norm{|ae_1|}{2}^2+\norm{|ae_2|}{2}^2=1\pl, \\ &\norm{b}{2}^2=\norm{e_1b}{2}^2+\norm{e_2b}{2}^2=\norm{|(e_1b)^*|}{2}^2+\norm{|(e_2b)^*|}{2}^2=1\end{align*}
where $|ae_1|,|(e_1b)^*|\in L_2(M_n(\M))\pl$ and $|ae_2|,|(e_2b)^*|\in L_2(M_m(\M))$. Then,
\begin{align*}
&\norm{x}{L_\infty^1(M_{n+m}(\N)\subset M_{n+m}(\M))}\\ \le & \norm{ae_2}{1}\norm{x_1}{L_1^\infty(M_{n}(\N)\subset M_{n}(\M))} \norm{e_1b}{2}+\norm{ae_2}{2}\norm{x_2}{L_1^\infty(M_{m}(\N)\subset M_{m}(\M))} \norm{e_2b}{2} \\ \le & \max\{\norm{x_1}{L_1^\infty(M_{n}(\N)\subset M_{n}(\M))} ,\norm{x_2}{L_1^\infty(M_{m}(\N)\subset M_{m}(\M))} \}\pl.
\end{align*}
Also the maximum in the inequality is achieved with $a,b\in L_2(M_{n}(\N))$ or $a,b\in L_2(M_{m}(\N))$. For $x\in M_n(\M), \al,\beta^*\in M_{n,m}$ , we have
\begin{align*}&\norm{(\al\ten 1)x(\beta\ten 1)}{L_1^\infty(M_{n}(\N)\subset M_{n}(\M))}
\\=&\sup_{\norm{\pl a\pl }{2}=\norm{\pl b\pl }{2}=1}\norm{a(\al\ten 1)x (\beta\ten 1)b}{L_1(M_{n}(\M))}\\\le & \sup_{\norm{\pl a\pl }{2}=\norm{\pl b\pl }{2}=1} \norm{a(\al\ten 1)}{2}\norm{x}{L_1^\infty(M_m(\N)\subset M_{m}(\M))} \norm{(\beta\ten 1)b}{2}
\\ \le&  \sup_{\norm{\pl a\pl }{2}=\norm{\pl b\pl }{2}=1}\norm{a}{2}\norm{\al}{M_{n,m}}\norm{x}{L_1^\infty(M_n(\N)\subset M_{n}(\M))} \norm{\beta}{M_{m,n}}\norm{b}{2}
\\ =&  \norm{\al}{M_{n,m}}\norm{x}{L_1^\infty(M_m(\N)\subset M_{m}(\M))} \norm{\beta}{M_{m,n}}\pl.
\end{align*}
Thus we verified $M_{n}(L_\infty^1(\N\subset\M)):= L_\infty^1(M_n(\N)\subset M_n(\M))$ indeed gives an operator space structure on $L_\infty^1(\N\subset\M)$. By complex interpolation, we obtain the operator space structure for $L_1^p(\N\subset\M)$.
\begin{prop}\label{os4}For $1\le p\le \infty$, we have isometric isomorphism
\begin{align*}&M_n(L_\infty^p(\N\subset \M))\cong L_\infty^p(M_n(\N)\subset M_n(\M))\pl.
\end{align*}
\end{prop}
\begin{proof}Recall the complex interpolation relation for $1\le p\le \infty$,
\begin{align*} L_\infty^{p}(\N\subset\M)=[L_\infty(\M),L_\infty^1(\N\subset\M)]_{1/p}=L_\infty^p(\N\subset\M)\pl.
\end{align*}
Note that $M_n(L_\infty(\M))\cong L_\infty(M_n(\M))$. Then by interpolation,
\begin{align*}
L_\infty^{p}(M_n(\N)\subset M_n(\M))\cong& [L_\infty(M_n(\M)),L_\infty^1(M_n(\N)\subset M_n(\M))]_{1/p}\\\cong&
[M_n(L_\infty(\M)),M_n(L_\infty^1(\N\subset \M))]_{1/p}
\\\cong&
M_n(
L_\infty^p(\N\subset\M))\pl. \qedhere
\end{align*}\end{proof}
The following lemma shows that the connection $L_1^p(\N\subset\M)$ norm can be attained by pairing with the positive elements of $L_\infty^{p'}(\N\subset \M)$ in the unit ball.
\begin{lemma}\label{sup}Let $\rho\in L_1(\M)$ be positive. We have
\begin{align*}\exp{\Big(\frac{1}{p'}D_p(\rho||\N)\Big)}=&\inf \{ \norm{\si}{L_1^p(\N\subset\M)} | \pl \rho \le \si \pl \text{for some positive}\pl \si\in L_1^p(\N\subset\M)\}\pl.
\\=& \sup \{tr(x\rho)\pl | \pl x\in \M_+,    \norm{x}{L_\infty^{p'}(\N\subset \M)}\le 1 \}\pl.
\end{align*}
If finite, they all equal to $\norm{\rho}{L_1^p(\N\subset\M)}$.  The equality also holds for $+\infty$.
\end{lemma}
\begin{proof}For the first inequality it is sufficient to show that \[\rho\le \si \pl \Longrightarrow\pl \norm{\rho}{L_1^p(\N\subset\M)}\le \norm{\si}{L_1^p(\N\subset\M)}\] Indeed, by $\rho\le \si $, we have $\rho=\si^{\frac{1}{2}}z\si^{\frac{1}{2}}$ for some $\norm{z}{\infty}\le 1$. Note that
\[\norm{\si}{L_1^p(\N\subset\M)}=\inf_{\si^{\frac{1}{2}}=a\eta} \norm{a}{L_{2p'}(\N)}^2\norm{\eta}{L_{2p}(\M)}^2\pl.\]
For each factorization $\si^{\frac{1}{2}}=a\eta$, we have $\rho=\si^{\frac{1}{2}}=a\eta z\eta^*a^*$ and hence
\begin{align*}
\norm{\rho}{L_1^p(\N\subset\M)}\le \norm{a}{L_{2p'}(\N)}\norm{\eta z\eta^*}{L_{p}(\M)}\norm{a^*}{L_{2p'}(\N)}\le \norm{a}{L_{2p'}(\N)}^2\norm{\eta}{L_{2p}(\M)}^2\pl.
\end{align*}
Thus $\norm{\rho}{L_1^p(\N\subset\M)}\le \norm{\si}{L_1^p(\N\subset\M)}$. For the second one we use a Grothendieck-Pietsch separation argument. Denote
\[\la(\rho):=\inf \{ \norm{\si}{L_p(\M)} | \rho \le \si \pl \text{for positive}\pl \si\in L_1^p(\N\subset\M)\}\]
We consider $\la(\rho)=+\infty$ if the infimum is empty. Let $\la$ be a positive number such that $\la<\la(\rho)$. Then for any positive $\si$ with $\norm{\si}{L_1^p(\N\subset\M)}\le  1$, we have
\[ \la \si-\rho \ngeq 0\pl,\]
hence has nontrivial negative part. Therefore there exists a positive $x\in \M_+$ such that $\norm{x}{\infty}=1$ and
\[ tr(\rho x )-\la tr(\si x) >0 \]
Consider the weak$^*$-compact subset
\[B=\{x\in  \M | \norm{x}{\infty}\le 1, x\ge 0\}\]
For each $\si\in \{\si \ge 0\pl, \norm{\si}{L_1^p(\N\subset\M)}\le 1\}$,
we define the function $f_\si: B\to \R$ as follows,
\[ f_{\si}(x)=tr (\rho x)-\la tr(\si x)\pl\]
These $f_{\si}$ are continuous with respect to weak$^*$-topology on $B$ because $\rho, \si\in L_1(\M)$. Denote
\begin{align*}&\mathcal{F}:=\{ f_{\si}\in C(B,\mathbb{R}) \pl |\si \ge 0\pl, \norm{\si}{L_1^p(\N\subset\M)}\le 1\pl.\}\\
&\F_-=\{ f\in C(B,\mathbb{R}) \pl |\pl  \sup f< 0\}
\end{align*}
Bothe $\F$ and $\F_-$ are convex and $\F_-$ is open. Moreover, $\F$ and $\F_-$ are disjoint because for each $f_\si\in \F$, $\sup_{x\in B} f_\si(x)>0$. Then by Hahn-Banach Theorem, there exists a norm one linear function $\psi: C(B,\mathbb{R})\to \mathbb{R}$ such that for any $f_-\in\F_-$ and $f_\si\in \F$,
\begin{align*}
\phi(f_-)<r \le \phi(f_\si)\pl.
\end{align*}
Because $\F_-$ is a cone, $r\ge 0$. Similarly, $r\le 0$ because for any $0<\delta<1$, $\delta\F\subset \F$. Then $r=0$ and $\phi$ is a positive linear functional because $\phi(f_-)< 0 $ for any $f_-\in \F_-$. By Riesz Representation Theorem, $\phi$ is given a Borel probablity measure $\mu$ on $B$. Namely.
\[\phi(f)=\int_B f(x)\mu(x)\pl.\]
Denote $x_0=\int_B x d\mu(x)$. We have for any positive $\si$ with $\norm{\si}{L_1^p(\N\subset\M)}\le 1$,
\begin{align*} \phi(f_\si)= \int_B f_\si(x)d\mu(x)&=\int_B tr(\rho x)-\la tr(\si x)d\mu(x)
=\tau(\rho x_0)-\la tr(\si x_0)\ge 0
\end{align*}
Note that for $x\in \M$,
\begin{align}
\sup_{\si \ge 0\pl, \norm{\pl\si\pl }{L_1^p(\N\subset\M)}\le 1} tr(\si x_0)&=\sup \{tr(aya^*x)\pl |\pl \norm{a}{L_{2p'}(\N)}\le 1,\pl y\ge 0,\pl \norm{y}{L_p(\M)}\le 1 \}\nonumber
\\ &=\sup \{\norm{a^*xa}{L_{p'}(\M)}\pl |\pl \norm{a}{L_{2p'}(\N)}\le 1 \}=\norm{x}{L_\infty^{p'}(\N\subset \M)}.
\end{align}
Thus, we have
\[\tau(\rho x_0)\ge \sup_{\si \ge 0\pl, \norm{\pl\si\pl }{L_1^p(\N\subset\M)}\le 1}\la tr(\si x_0)=\la \norm{x}{L_\infty^{p'}(\N\subset \M)}\pl.\]
By linearity, we prove that \[\sup \{tr(x\rho)\pl | \pl x\in \M, \pl x\ge 0,   \norm{x}{L_\infty^{p'}(\N\subset \M)}=1 \}\ge \la(\rho).\] If $\la(\rho)=+\infty$, they are clearly equal. If $\la(\rho)$ is finite, \[\la(\rho)=\norm{\rho}{L_1^p(\N\subset\M)}\ge \sup \{tr(x\rho)\pl | \pl x\in \M, \pl x\ge 0,   \norm{x}{L_\infty^{p'}(\N\subset \M)}=1 \}\pl,\]
by the duality $L_\infty^{p'}(\N\subset \M)\subset L_1^p(\N\subset\M)^*$. That completes the proof.
\end{proof}

\bibliography{asymmetry1}
\bibliographystyle{plain}

\end{document}